\documentclass{article}
\usepackage{amssymb}
\usepackage{amsmath}
\usepackage{amsfonts}
\usepackage{color}
\usepackage{graphicx}
\usepackage{caption}
\usepackage{subcaption}
\usepackage{amsthm}

\theoremstyle{plain}

\newtheorem{theorem}{Theorem}[section]
\newtheorem{proposition}[theorem]{Proposition}

\newtheorem{remark}[theorem]{Remark}

\theoremstyle{definition}
\newtheorem{definition}[theorem]{Definition}

\newcommand{\bydef}{\,\stackrel{\mbox{\tiny\textnormal{\raisebox{0ex}[0ex][0ex]{def}}}}{=}\,}
\newcommand{\bx}{\bar{x}}

\newcommand{\bu}{\bar{u}}
\newcommand{\tu}{\tilde{u}}
\newcommand{\N}{\mathbb{N}}
\newcommand{\R}{\mathbb{R}}
\newcommand{\C}{\mathbb{C}}

\def\corcommstyle{\bf\small\tt}

\def\corrl #1<<#2||#3>>{
\if\visiblecomments y
  \begin{quote} {\corcommstyle $<<$COMMENT$>>$ {\color{red}#1\marginpar{!!}}\\$<<$OLD$<<$} \end{quote}

{\color{red} 
 #2
 }

  \begin{quote} {\corcommstyle ==NEW== } \end{quote}
   \noindent\hrulefill
 
\vspace{-10pt} 
 
 \noindent\hrulefill
 
 \vspace{-10pt} 
 
 \noindent\dotfill
 
  #3
  
   \noindent\dotfill 

\vspace{-10pt} 
 
 \noindent\hrulefill
 
 \vspace{-10pt} 
 
 \noindent\hrulefill
  \begin{quote} {\corcommstyle $>>$END$>>$ } \end{quote}
 \else
  #3
 \fi
}

\long\def\longcorrl #1<<#2||#3>>{
\if\visiblecomments y
  \begin{quote} {\corcommstyle $<<$COMMENT$>>$ {\color{red}#1\marginpar{!!}}\\$<<$OLD$<<$} \end{quote}
 
 {\color{red}

  #2
  
  }
  
  \begin{quote} {\corcommstyle ==NEW== } \end{quote}
  
    \noindent\hrulefill
 
\vspace{-10pt} 
 
 \noindent\hrulefill
 
 \vspace{-10pt} 
 
 \noindent\dotfill
 
  #3
  
   \noindent\dotfill 

\vspace{-10pt} 
 
 \noindent\hrulefill
 
 \vspace{-10pt} 
 
 \noindent\hrulefill
  \begin{quote} {\corcommstyle $>>$END$>>$ } \end{quote}
 \else
  #3
 \fi
}

\def\mlabel #1
{
  \if\visiblecomments y
     \marginpar[\flushright \bf \footnotesize #1]{\bf \footnotesize #1}
  \fi
}

\def\flabel #1
{
  \if\visiblecomments y
       \hbox{\bf\footnotesize #1}
  \fi
}

\def\corrq #1<<#2>>{
\if\visiblecomments y
  \begin{quote} {\corcommstyle $<<$COMMENT$>>$ #1\marginpar{!!}\\$<<$BEG$<<$} \end{quote}
  \noindent\hrulefill
 
\vspace{-10pt} 
 
 \noindent\hrulefill
 
 \vspace{-10pt} 
 
 \noindent\dotfill
 
 {\color{red}
  
  #2
  
  }
   
  \noindent\dotfill 

\vspace{-10pt} 
 
 \noindent\hrulefill
 
 \vspace{-10pt} 
 
 \noindent\hrulefill 
  \begin{quote} {\corcommstyle $>>$END$>>$ } \end{quote}
 \else
  #2
 \fi
}

\long\def\longcorrq #1<<#2>>{
\if\visiblecomments y
  \begin{quote} {\corcommstyle $<<$COMMENT$>>$ #1\marginpar{!!}\\$<<$BEG$<<$} \end{quote}
  \noindent\hrulefill
 
\vspace{-10pt} 
 
 \noindent\hrulefill
 
 \vspace{-10pt} 
 
 \noindent\dotfill

  #2

  \noindent\dotfill 

\vspace{-10pt} 
 
 \noindent\hrulefill
 
 \vspace{-10pt} 
 
 \noindent\hrulefill 
  \begin{quote} {\corcommstyle $>>$END$>>$ } \end{quote}
 \else
  #2
 \fi
}

\def\corrc #1<<>>{
\if\visiblecomments y
  \begin{quote} {\corcommstyle $<<$COMMENT$>>$ \color{red} #1\marginpar{!!}} \end{quote}
\fi
}

\def\corre #1<<#2||#3>>{
\if\visiblecomments y
  #3\marginpar{\corcommstyle #1}
 \else
  #3
 \fi
}

\long\def\longcorre #1<<#2||#3>>{
\if\visiblecomments y
  #3\marginpar{\corcommstyle #1}
 \else
  #3
 \fi
}

\def\corrs #1<<#2||#3>>{
\if\visiblecomments y
  #3\marginpar{\corcommstyle #2 $\rightarrow$ #3\\ #1}
 \else
  #3
 \fi
}

\def\corro #1<<#2||#3>>{
#2}

\def\corrn #1<<#2||#3>>{
#3}

\long\def\longcorro #1<<#2||#3>>{
#2}

\long\def\longcorrn #1<<#2||#3>>{
#3}

\long\def\underconstruction #1<<<#2>>>{
\if\visiblecomments y
  \begin{quote} {\corcommstyle $<<$UNDER CONSTRUCTION - BEGIN$>>$ #1\marginpar{!!}} \end{quote}
  #2
  \begin{quote} {\corcommstyle $>>$UNDER CONSTRUCTION - END$>>$ } \end{quote}
 \else
 \fi
}

\def\showcomments{
  \let\visiblecomments y
}

\def\hidecomments{
  \let\visiblecomments n
}

\showcomments
\begin{document}

\title{Spatial relative equilibria and periodic solutions of the Coulomb $(n+1)$%
-body problem\thanks{Data sharing not applicable to this article as no datasets were generated or analysed during the current study}}
\author{ Kevin Constantineau \thanks{%
McGill University, Department of Mathematics and Statistics, 805 Sherbrooke
Street West, Montreal, QC, H3A 0B9, Canada. \texttt{%
kevin.constantineau@mail.mcgill.ca}} 
\and 
Carlos Garc\'{\i}a-Azpeitia\thanks{%
IIMAS, Universidad Nacional Aut\'onoma de M\'exico, Apdo. Postal 20-726,
C.P. 01000, M\'exico D.F., M\'exico. \texttt{%
cgazpe@ciencias.unam.mx}} 
\and
Jean-Philippe Lessard\thanks{%
McGill University, Department of Mathematics and Statistics, 805 Sherbrooke
Street West, Montreal, QC, H3A 0B9, Canada. \texttt{jp.lessard@mcgill.ca}} }
\maketitle

\begin{abstract}

We study a classical model for the atom that considers the movement of $n$
charged particles of charge $-1$ (electrons) interacting with a fixed
nucleus of charge $\mu >0$. We show that two global branches of spatial
relative equilibria bifurcate from the $n$-polygonal relative equilibrium
for each critical values $\mu =s_{k}$ for $k\in \lbrack 2,...,n/2]$. In these
solutions, the $n$ charges form $n/h$-groups of regular $h$-polygons in space, where $h$ is the greatest common divisor of $k$ and $n$. Furthermore, each spatial relative equilibrium has a global branch of relative periodic solutions for each normal frequency satisfying some nonresonant condition. We obtain computer-assisted proofs of existence of several spatial relative equilibria on global branches away from the $n$-polygonal relative equilibrium. Moreover, the nonresonant condition of the normal frequencies for some spatial relative equilibria is verified rigorously using computer-assisted proofs.


\end{abstract}

\section{Introduction}

The Thomson problem is a classical model to study a configuration of $n$
electrons, constrained to the unit sphere, that repel each other with a
force given by Coulomb's law. Thomson posed the problem in 1904 as an atomic
model, later called the \emph{plum pudding model} \cite{LAFAVE20131029}.
Without loss of generality we can assume that the elementary charge of an
electron is $e=-1$, its mass is $m=1$, and the Coulomb constant is $1$. We
wish to analyze another classical model for the atom that considers the
movement of $n$ charged particles with negative charge $-1$ (electrons)
interacting with a fixed nucleus with positive charge $\mu $. Since
electrons and protons have equal charges with different signs, for a
non-ionized atom we consider that $\mu =n$. By supposing that the
gravitational forces are smaller than Coulomb's forces, the system of
equations describing the movement of the charges is%
\begin{equation}
\ddot{q}_{j}=-\mu \frac{q_{j}}{\left\Vert q_{j}\right\Vert ^{3}}%
+\sum_{i=0~(i\neq j)}^{n-1}\frac{q_{j}-q_{i}}{\left\Vert
q_{j}-q_{i}\right\Vert ^{3}},\qquad q_{j}\in \mathbb{R}^{3},\qquad
j=0,...,n-1,  \label{Eq}
\end{equation}%
where the first term of the force represents the interaction with the fixed
nucleus.

This problem is referred to as the charged $(n+1)$-body problem \cite%
{MR3007103,AlPe02}, the $n$-electron atom problem \cite{MR726510} and the
Coulomb $(n+1)$-body problem \cite{MR4032352} (and the references therein).
An interesting feature of this problem is the existence of spatial relative
equilibria, in contrast to the $n$-body problem where all relative
equilibria must be planar \cite{AlPe02}. In this paper we investigate the
existence of bifurcations of spatial relative equilibria arising from the
polygonal relative equilibrium. The existence of bifurcations of spatial
central configurations arising from planar configurations has been
investigated previously in \cite{Moeckel_1995}.

Specifically, we study equation (\ref{Eq}) in rotating coordinates $u_{j}$
with frequency $\sqrt{\omega }$, $q_{j}(t)=e^{\sqrt{\omega }t\bar{J}}u_{j}(t)
$, where $\bar{J}=J\oplus 0$ and $J$ is the standard symplectic matrix in $%
\mathbb{R}^{2}$. The starting solution of our study is the planar unitary
polygon $u=a$, which is an equilibrium of the equations in a rotating frame
with $\omega =\mu -s_{1}>0$, where 
\[
s_{k}\,\overset{\mbox{\tiny\textnormal{\raisebox{0ex}[0ex][0ex]{def}}}}{=}\,%
\frac{1}{4}\sum_{j=1}^{n-1}\frac{\sin ^{2}(kj\pi /n)}{\sin ^{3}(j\pi /n)}%
\text{,\qquad }k=0,...,n-1.
\]%
In Theorem \ref{Thm1} we prove that for each $k\in \lbrack 2,n/2]\cap 
\mathbb{N}$, there are two \emph{global bifurcations} of stationary
solutions (in rotating frame) from the trivial solution $u={a}$ at $\mu
=s_{k}$. Furthermore, each solution in the families is a spatial relative
equilibrium where the $n$ charges form $n/h$-groups of regular  $h$-polygons
in the space, where $h$ is the greatest common divisor of $k$ and $n$ (this
fact can be observed in the numerical computations in Figures \ref%
{fig:example1}, \ref%
{fig:example2}, \ref{fig:all_plots} and \ref{fig:all_plots2}). By \emph{%
global bifurcation} we mean that the bifurcation forms a connected component
that either goes to other relative equilibria, ends in a collision
configuration or its norm goes to infinity. The global property is proved
using Brouwer degree along the lines of the result in \cite{MR2832692},
which treats the bifurcation of relative equilibria from the unitary polygon
in the $n$-body problem.

\begin{figure}[tbp]
\centering
\begin{subfigure}[b]{0.45\textwidth}
         \centering
         \includegraphics[width=\textwidth]{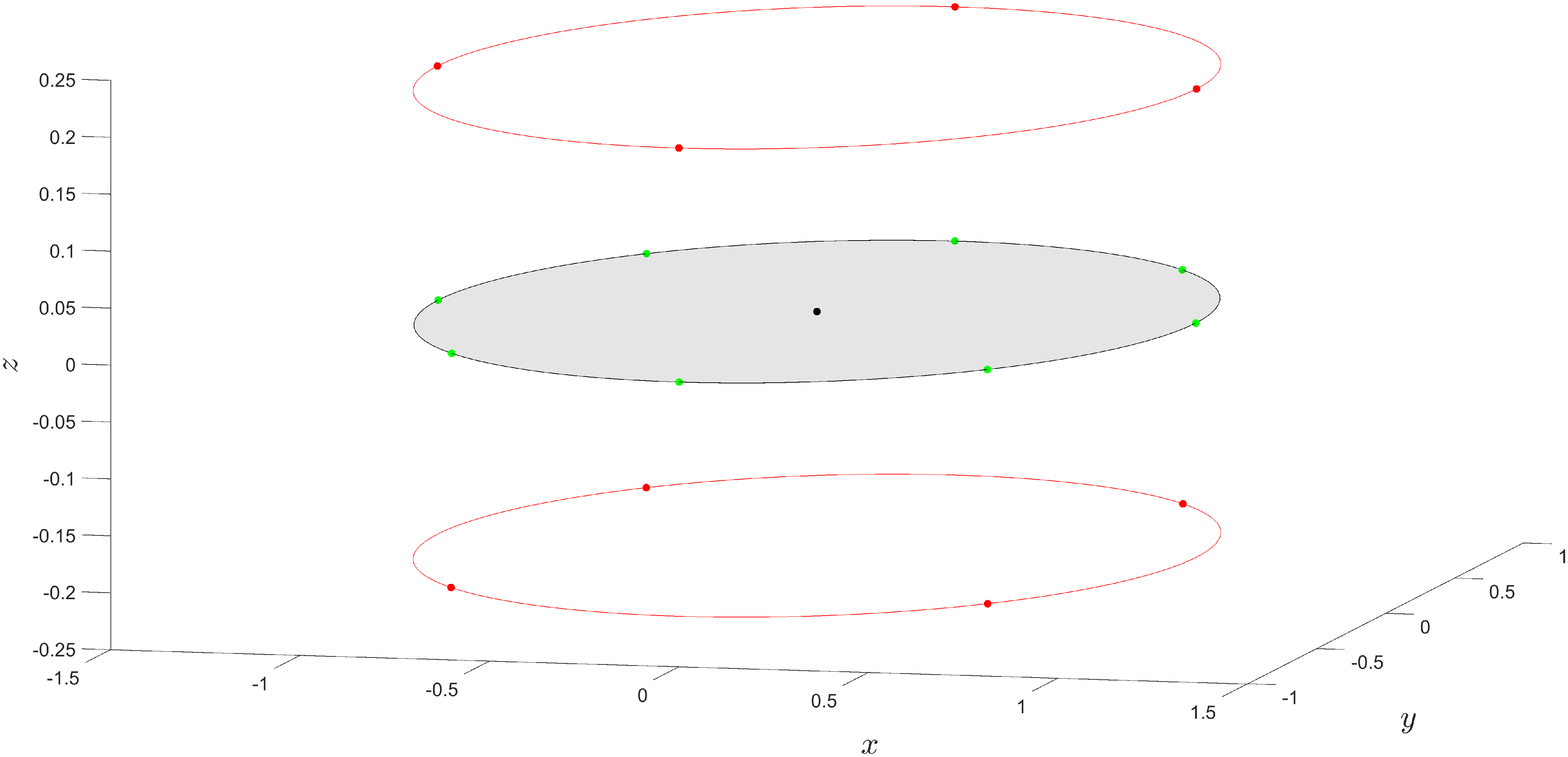}
         \label{fig:1}
     \end{subfigure}
\hfill 
\begin{subfigure}[b]{0.45\textwidth}
         \centering
         \includegraphics[width=\textwidth]{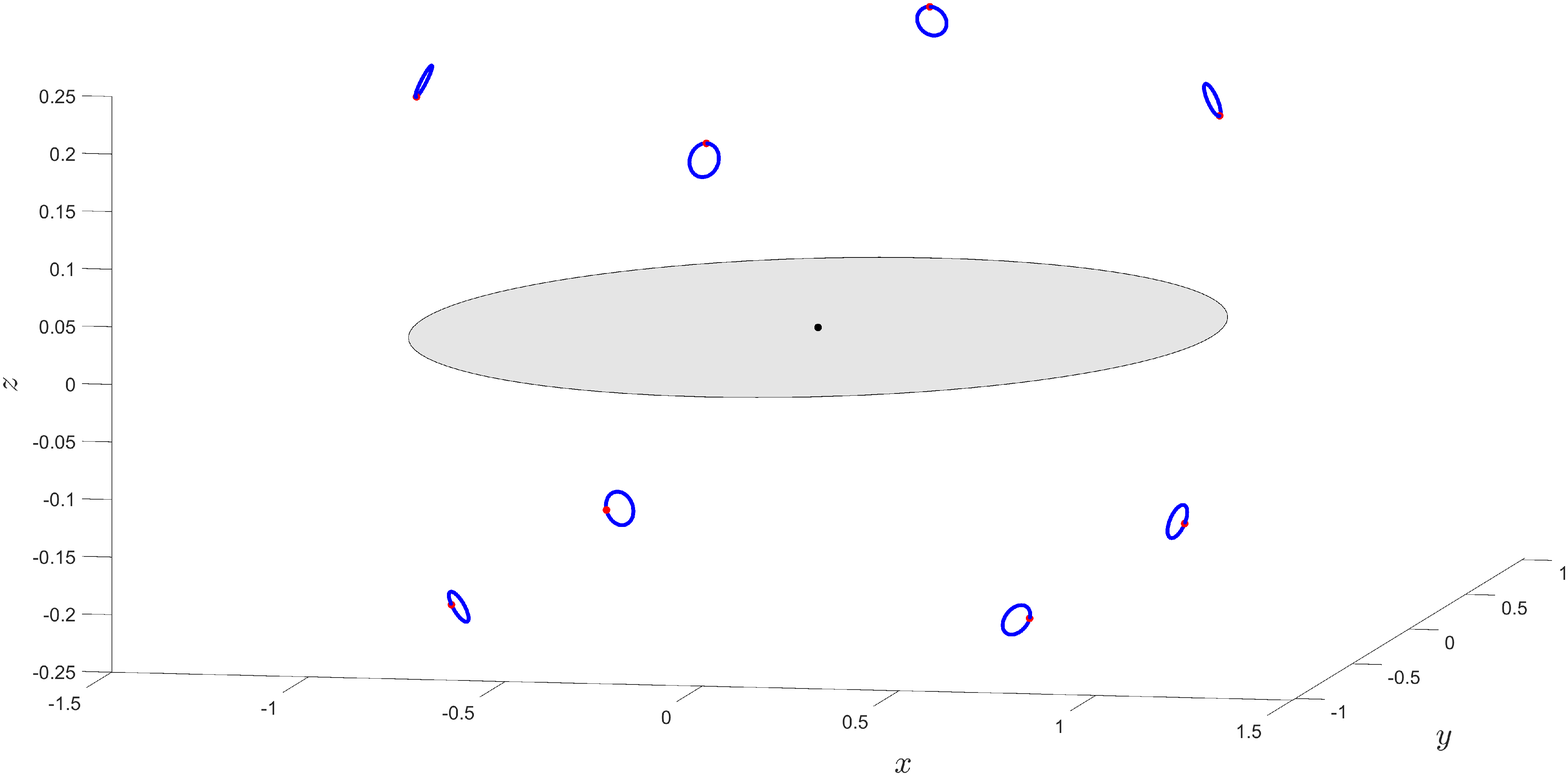}
         \label{fig:2}
     \end{subfigure}
\caption{\textbf{Left:} A relative equilibrium solutions
for $n=8$ and $k=4$ in the first family, whose existence has been obtained
using a computer-assisted proof (together with tight error bounds). \textbf{%
Right:} The first order expansion of the family of periodic solutions arising from the relative equilibrium in a rotating frame.}
\label{fig:example1}
\end{figure}

\begin{figure}[tbp]
\centering
\begin{subfigure}[b]{0.45\textwidth}
         \centering
         \includegraphics[width=\textwidth]{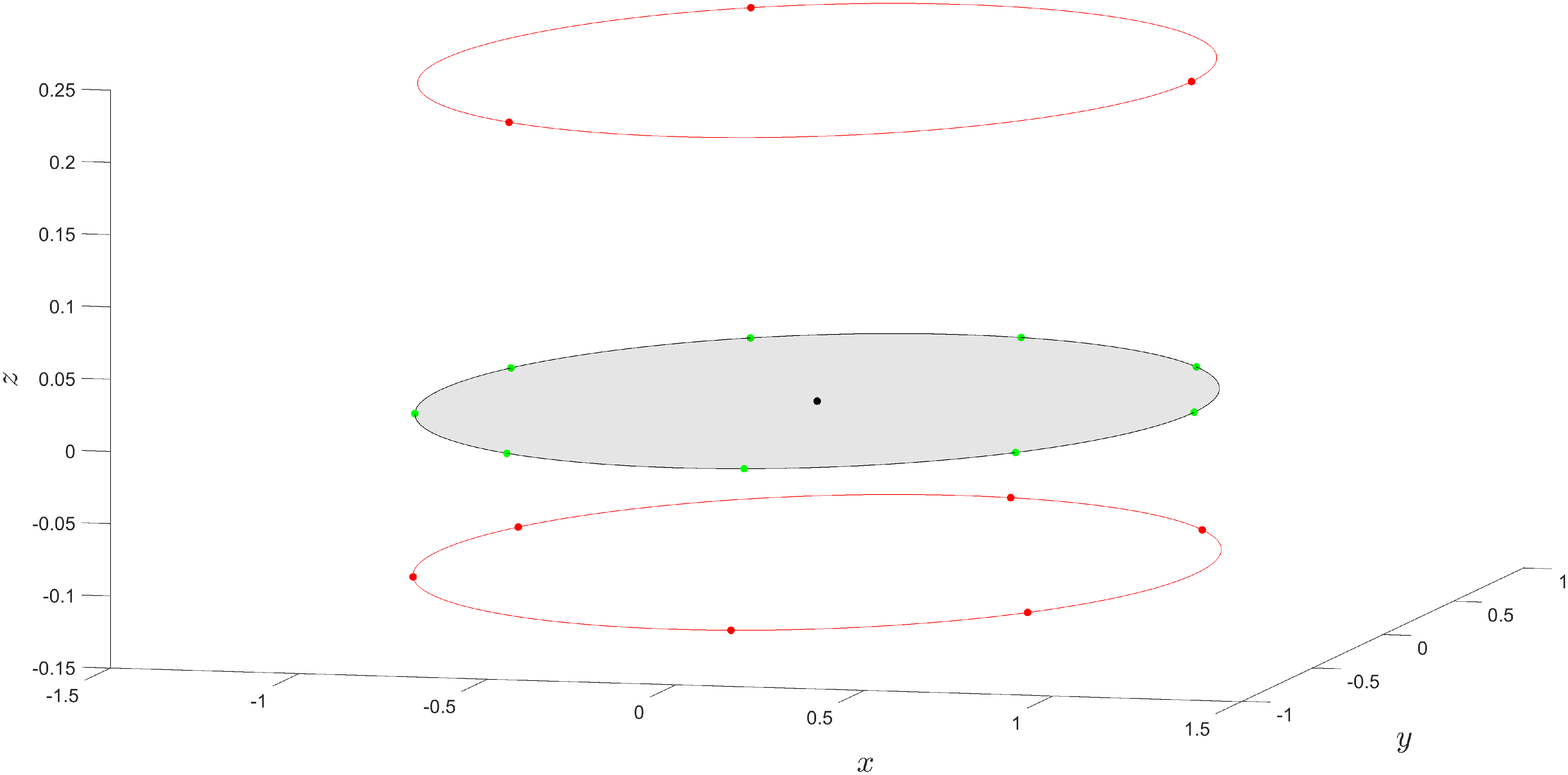}
         \label{fig:1}
     \end{subfigure}
\hfill 
\begin{subfigure}[b]{0.45\textwidth}
         \centering
         \includegraphics[width=\textwidth]{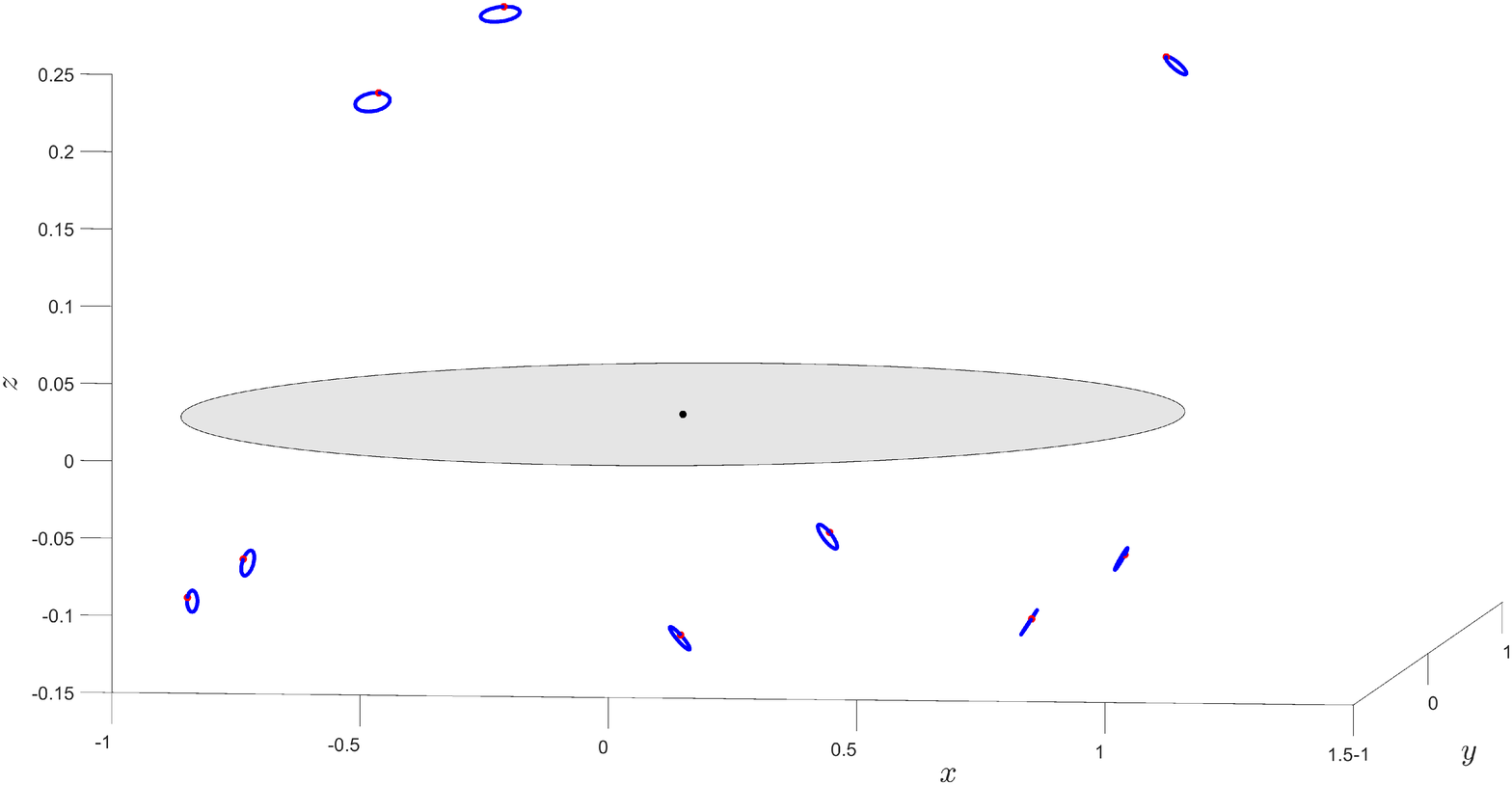}
         \label{fig:2}
     \end{subfigure}
\caption{\textbf{Left:} A relative equilibrium solutions
for $n=9$ and $k=3$ in the first family, whose existence has been obtained
using a computer-assisted proof (together with tight error bounds). \textbf{%
Right:} The first order expansion of the family of periodic solutions arising from the relative equilibrium in a rotating frame.}
\label{fig:example2}
\end{figure}

In the rotating frame, the linearized system at a spatial relative
equilibrium exhibits many periodic solutions (normal modes). In the present
paper, we prove the persistence of these periodic solutions in the nonlinear
system (see Theorem~\ref{Thm2}). These solutions are referred to as
nonlinear normal modes or Lyapunov families. In the inertial frame they are
known as relative periodic solutions and correspond to periodic or
quasiperiodic solutions. Furthermore, we prove the \emph{global property},
which in the present context means that a family of periodic solutions,
represented by a continuous branch $\mathcal{C} $ in the space of
frequencies and $2\pi $-periodic paths, is not compact or comes back to
another bifurcation point. The non-compactness of $\mathcal{C}$ implies that
either the norm or period of the solutions from $\mathcal{C}$ goes to
infinity or $\mathcal{C}$ ends in a collision orbit. The proof of the global
property is akin to the proof in \cite{MR2832692} regarding the existence of
periodic solutions from the polygonal relative equilibrium $a$, and is
obtained by means of the $SO(2)\times S^{1}$-equivariant degree developed in 
\cite{MR1984999}.

The result of Theorem \ref{Thm2} holds when some non-resonance assumptions
(Definition \ref{Def1}) on the normal frequencies of the spatial relative
equilibrium are satisfied. A main contribution of the present paper is the
implementation of computer-assisted proofs to validate global branches of
spatial relative equilibria and also the non-resonance assumption, which is
required to obtain the existence of families of periodic solutions arising
from them. In Section 3, we present the general approach (a
Newton-Kantorovich type theorem, see Theorem~\ref{thm:newton-kantorovich})
used to obtain the different computer-assisted proofs. This theorem is used
to validate the spatial relative equilibria (in Section~\ref{sec:CAP_rel_eq}%
) and its normal frequencies (in Section~\ref{sec:CAP_frequencies}) to
verify the conditions of Theorem~\ref{Thm2}. Figures \ref{fig:example1} and \ref%
{fig:example2}
contain an example of spatial relative equilibria whose existence has been
obtained using a computer-assisted proof (together with tight error bounds)
and for which the non-resonance condition has been rigorously verified.
Similar computer-assisted proofs were carried on for each (non black) point
in Figures~\ref{fig:all_plots} and~\ref{fig:all_plots2}.

\section{Existence of periodic solutions arising from spatial relative
equilibria}

We assume that: the gravitational forces are much smaller than Coulomb's
forces, the charge in the center is $\mu >0$, the $n$ charges have charge $%
-1 $, the mass of the charges is $m=1$ and the Coulomb constant is $\kappa =1
$. We also assume that the position of the central charge is fixed at the
center and the positions of the $n$ charges are determined by $%
q(t)=(q_{0}(t),...,\,q_{n-1}(t))$ with $q_{j}(t)\in \mathbb{R}^{3}$ for $%
j=0,...,n-1$. Under these assumptions, the system satisfies the Newtonian
equation

\begin{equation}
\ddot{u}(t)=\nabla U(q(t)),  \label{eqn01}
\end{equation}%
where $U(q)$ is the potential energy given by 
\[
U(q)\,\overset{\mbox{\tiny\textnormal{\raisebox{0ex}[0ex][0ex]{def}}}}{=}%
\,\sum_{j=0}^{n-1}\frac{\mu }{\Vert q_{j}\Vert }-\sum_{i<j}\frac{1}{\Vert
q_{j}-q_{i}\Vert },
\]%
where the first term represents the interaction with the fixed center.

Let $\bar{J}=J\oplus 0$, where $J$ is the standard symplectic matrix in $%
\mathbb{R}^{2}$. In rotating coordinates, $q_{j}(t)=e^{\sqrt{\omega }t\bar{J}%
}u_{j}(t)$, the system of equations becomes $\ddot{u}_{j}+2\sqrt{\omega }%
\bar{J}\dot{u}_{j}=\nabla _{u_{j}}V(u)$, where $V$ is the augmented
potential 
\[
V(u)=\frac{\omega }{2}\sum_{j=0}^{n-1}\left\Vert \bar{I}u_{j}\right\Vert
^{2}+\sum_{j=0}^{n-1}\frac{\mu }{\left\Vert u_{j}\right\Vert }-\sum_{i<j}%
\frac{1}{\left\Vert u_{j}-u_{i}\right\Vert }\text{,}
\]%
where $\bar{I}=-\bar{J}^{2}=1\oplus 1\oplus 0$. Let $u=(u_{0},...,u_{n-1})$
and $\mathcal{\bar{J}}=\bar{J}\oplus ...\oplus \bar{J}$, the system of
equation reads%
\begin{equation}
\ddot{u}+2\sqrt{\omega }\mathcal{\bar{J}}\dot{u}=\nabla V(u)\text{.}
\label{Equ}
\end{equation}

Let $S_{n}$ be the group of permutations of $\{0,1,...,n-1\}$ and $D_{n}$ be
the subgroup generated by the permutations $\zeta (j)=j+1$ and $\kappa
(j)=n-j$ mod $n$. We define the action of $\gamma \in S_{n}$ in $\mathbb{R}%
^{3n}$ as%
\begin{equation}
\rho (\gamma )(u_{0},u_{1},...,u_{n-1})=(u_{\gamma (0)},u_{\gamma
(1)},...,u_{\gamma (n-1)}).  \label{permutation}
\end{equation}%
Notice that this is a left action only if the product on $S_{n}$ is defined
according to the opposite convention that $\sigma _{1}\sigma _{2}\,\overset{%
\mbox{\tiny\textnormal{\raisebox{0ex}[0ex][0ex]{def}}}}{=}\,\sigma _{2}\circ
\sigma _{1}$. Clearly the potential $V$ is $S_{n}$-invariant. On the other
hand, while the potential $U(u)$ is $O(3)$-invariant, the potential $V(u)$
is only invariant under the action of the normalizer $O(2)\times \mathbb{Z}%
_{2}$ of $SO(2)\subset O(3)$.

We conclude that $V(u)$ is $G$-invariant with 
\[
G \,\overset{\mbox{\tiny\textnormal{\raisebox{0ex}[0ex][0ex]{def}}}}{=}\,
S_{n}\times O(2)\times \mathbb{Z}_{2}.
\]%
The explicit action of the elements $\theta ,\kappa _{y}\in O(2)$ and $%
\kappa _{z}\in \mathbb{Z}_{2}$ in the components of $u$ is given by 
\begin{equation}
\theta u_{j}=e^{-\mathcal{\bar{J}}\theta }u_{j},\qquad \kappa
_{y}u_{j}=R_{y}u_{j},\qquad \kappa _{z}u_{j}=R_{z}u_{j},  \nonumber
\end{equation}%
where 
\[
R_{y}=1\oplus -1\oplus 1,\qquad R_{z}=1\oplus 1\oplus -1.
\]

\subsection{The polygonal relative equilibrium}

The polygon $a=(a_{0},...,a_{n-1})$, where 
\[
a_{j}=\left( e^{ij\zeta },0\right) \in \mathbb{C\times R},\qquad \zeta \,%
\overset{\mbox{\tiny\textnormal{\raisebox{0ex}[0ex][0ex]{def}}}}{=}\,2\pi /n,
\]%
is a critical point of $V(u;\omega )$ for $\omega =\mu -s_{1}>0$. This
follows from the identity%
\[
\nabla _{u_{j}}V(a)=\omega a_{j}-\mu a_{j}+\sum_{i=0~(i\neq j)}^{n-1}\frac{%
a_{j}-a_{i}}{\left\Vert a_{j}-a_{i}\right\Vert ^{3}}=a_{j}\left( \omega -\mu
+s_{1}\right) ,
\]%
where we have used that 
\[
\sum_{i=0~(i\neq j)}^{n-1}\frac{a_{j}-a_{i}}{\left\Vert
a_{j}-a_{i}\right\Vert ^{3}}=a_{j}\sum_{i=0~(i\neq j)}^{n-1}\frac{%
1-e^{ij\zeta }}{\left\Vert 1-e^{ij\zeta }\right\Vert ^{3}}=a_{j}\frac{1}{4}%
\sum_{j=1}^{n-1}\frac{1}{\sin (j\pi /n)}=a_{j}s_{1}\text{.}
\]%
\subsection{Bifurcation of spatial relative equilibria}

Hereafter we assume that $\omega =\mu -s_{1}$ and we denote the dependence
of the potential $V$ in the parameter $\mu $ as $V(u;\mu )$. Thus the
polygon $u=a$ is a trivial solution of $\nabla V(u;\mu )=0$ with isotropy
group $G_{a}$ generated by 
\[
\tilde{\zeta}=\left( \zeta ,\zeta ,e\right) ,\qquad \tilde{\kappa}%
_{y}=\left( \kappa ,\kappa _{y},e\right) ,\qquad \tilde{\kappa}_{z}=\left(
e,e,\kappa _{z}\right) \in S_{n}\times O(2)\times \mathbb{Z}_{2},
\]%
where $e$ represents the identity element. As a consequence of the
continuous action of $SO(2)$, the orbit of the polygonal equilibrium $a$ is
one dimensional. Thus the generator of the $SO(2)$-orbit of $a$ is $-\bar{%
\mathcal{J}}a$, which belongs to the kernel of $D^{2}V(u)$.

Given that $G_{a}$ fixes $a$, then $D^{2}V(a)$ is $G_{a}$-equivariant and,
by Schur's lemma, it has the same eigenvalue in each irreducible
representations under the action of $G_{a}$. The spatial irreducible
representations of $G_{a}$ are obtained in section \textquotedblleft 8.4.
The problem of $n$-charges\textquotedblright\ in the paper \cite{MR3007103}.
Define $V_{k}$ as the subspace generated by 
\[
v_{k}=\left( v_{k}^{0},,...,v_{k}^{n-1}\right) \in \mathbb{R}^{3n},
\]%
where $v_{k}^{j}=(0,0,\cos jk\zeta )$ for $k\in \left[ 0,n/2\right] \cap 
\mathbb{N}$ and $v_{k}^{j}=(0,0,\sin jk\zeta )$ for $k\in (n/2,n-1]\cap 
\mathbb{N}$, then the irreducible $G_{a}$-representations are given by $%
V_{k} $ for $k=0,n/2$ and $V_{k}\oplus V_{n-k}$ for $k\in \lbrack 1,n/2)\cap 
\mathbb{N}$.

Specifically, using the isomorphism 
\[
av_{k}+bv_{n-k}\in V_{k}\oplus V_{n-k}\rightarrow a+ib\in \mathbb{C},
\]%
the action of $\tilde{\zeta},\tilde{\kappa}_{y},\tilde{\kappa}_{z}\in G_{a}$
in $z=a+ib\in \mathbb{C}$ is given by 
\begin{equation}
\tilde{\zeta}z=e^{ik\zeta }z,\qquad \tilde{\kappa}_{y}z=\bar{z},\qquad 
\tilde{\kappa}_{z}z=-z.  \label{Ac}
\end{equation}%
For $k=0,n/2$ the action in $V_{k}\simeq \mathbb{R}$ is the same as before
but with $z\in \mathbb{R}$. For instance, for $k\in (0,n/2)\cap \mathbb{N}$
we have that 
\begin{align*}
\tilde{\zeta}\left( av_{k}+bv_{n-k}\right) & =e^{-\mathcal{\bar{J}}\zeta
}\zeta \cdot \left( av_{k}+bv_{n-k}\right) \\
& =\left\{ \left( 0,0,a\cos (jk+k)\zeta -b\sin (jk+k)\zeta \right) \right\}
_{j=1}^{n} \\
& =\left( a\cos k\zeta -b\sin k\zeta \right) v_{k}+\left( b\cos k\zeta
+a\sin k\zeta \right) v_{n-k}.
\end{align*}%
Hence, we obtain that 
\[
\tilde{\zeta}z=\left( a\cos k\zeta -b\sin k\zeta \right) +i\left( b\cos
k\zeta +a\sin k\zeta \right) =e^{ik\zeta }z.
\]

In section \textquotedblleft 8.4. The problem of $n$-charges%
\textquotedblright\ in the paper \cite{MR3007103} is proven that the
eigenvalue of Hessian $D^{2}V(a)$ in each irreducible representation $V_{k}$
for $k=0,n/2$ and $V_{k}\oplus V_{n-k}$ is\ $-\mu +s_{k}$. For sake of
completeness we present a short proof of this fact. For $k\in \{0,...,n-1\}$%
, we define $T_{k}:\mathbb{R}\rightarrow W_{k}$ as%
\begin{align*}
T_{k}(w)& =(0,0,n^{-1/2}e^{ik\zeta }w,...,0,0,n^{-1/2}e^{nik\zeta }w)\text{
with } \\
W_{k}& =\{(0,0,e^{ik\zeta }w,...,0,0,e^{nik\zeta }w)\in \mathbb{C}^{3n}:w\in 
\mathbb{R}\}\text{.}
\end{align*}%
Since $W_{k}=V_{k}\oplus iV_{n-k}$,\ the result follows from the invariance
of the subspaces $W_{k}$ of $D^{2}V(a)$ with the following computation.

\begin{proposition}
For $k\in \{0,...,n-1\}$, we have 
\[
D^{2}V(a)T_{k}(w)=T_{k}(\left( -\mu +s_{k}\right) w).
\]
\end{proposition}

\begin{proof}
Let $\mathcal{A}_{ij}$ be the $3\times 3$ minor blocks of the Hessian $%
D^{2}V(a)=(\mathcal{A}_{ij})_{ij=1}^{n}$. The fact that $a$ is a planar
configuration implies that the matrices $\mathcal{A}_{ij}$ are block diagonal%
\begin{equation*}
\mathcal{A}_{ij}=diag(A_{ij},a_{ij})\text{,}
\end{equation*}%
where $A_{ij}$ is a $2\times 2$ matrix. For our purpose we only need to
compute the numbers $a_{ij}$. Let $d_{ij}=\left\vert u_{i}-u_{j}\right\vert $
be the distance between $u_{i}=(x_{i},y_{i},z_{i})$ and $%
u_{j}=(x_{j},y_{j},z_{j})$. For $i\neq j$ we have that $\partial
_{z_{i}}(d_{ij}^{-1})=-\partial _{z_{j}}(d_{ij}^{-1})$ and%
\begin{equation}
a_{ij}=-\partial _{z_{j}}\partial _{z_{i}}d_{ij}^{-1}|_{u=a}=\partial
_{z_{i}}^{2}d_{ij}^{-1}|_{u=a}=-d_{ij}^{-3}|_{u=a}=-\left( 2\sin ((i-j)\zeta
/2)\right) ^{-3}\text{.}  \label{aij}
\end{equation}%
For $i=j$ the number $a_{ii}$ satisfies%
\begin{equation*}
a_{ii}=\partial _{z_{i}}^{2}\left( \frac{\mu }{\left\Vert u_{j}\right\Vert }%
\right) _{u_{j}=a_{j}}-\sum_{j=0~(j\neq i)}^{n-1}\left( \partial
_{z_{i}}^{2}d_{ij}^{-1}\right) _{u_{j}=a_{j}}=-\mu -\sum_{j=0~(j\neq
i)}^{n-1}a_{ij}\text{.}
\end{equation*}

Now we need to denote to the component $w_{i}\in \mathbb{C}^{3}$ of the
vector $w=(w_{0},...,w_{n-1})\in \mathbb{C}^{3n}$ as $[w]_{i}=w_{i}$. From
the definitions we have that%
\begin{equation*}
\lbrack D^{2}V(a)T_{k}(w)]_{l}=\left(
0,0,n^{-1/2}\sum_{j=0}^{n-1}a_{lj}e^{ijk\zeta }w\right) \text{.}
\end{equation*}%
Since $a_{lj}=a_{0(j-l)}$ with $(j-l)\in \{0,...,n-1\}$ modulus $n$. From
the equality $a_{lj}e^{ijk\zeta }=e^{ilk\zeta }\left(
a_{0(j-l)}e^{i(j-l)k\zeta }\right) $ we have that%
\begin{equation*}
\lbrack D^{2}V(a)T_{k}(w)]_{l}=\left( 0,0,n^{-1/2}e^{ilk\zeta }b_{k}w\right)
=\left[ T_{k}(b_{k}w)\right] _{l}\text{,}
\end{equation*}%
where 
\begin{equation*}
b_{k}=\sum_{j=0}^{n-1}a_{0j}e^{ijk\zeta }=-\mu +\sum_{j=1}^{n-1}\left(
e^{ijk\zeta }-1\right) a_{0j}\text{,}
\end{equation*}%
because $a_{00}=-\mu -\sum_{j=0}^{n-1}a_{0j}$. Finally, using (\ref{aij}),
we obtain that%
\begin{equation*}
\sum_{j=1}^{n-1}\left( e^{ijk\zeta }-1\right) a_{0j}=\sum_{j=1}^{n-1}\frac{%
2\sin ^{2}(kj\zeta /2)}{2^{3}\sin ^{3}(j\zeta /2)}=s_{k}\text{.}
\end{equation*}
\end{proof}

According to \cite{MR3007103}, for $k\in \lbrack 2,n/2)$, the Hessian $%
D^{2}V(u;s_{k})$ has no additional zero-eigenvalues to the double
zero-eigenvalue corresponding to the subspace $V_{k}\oplus V_{n-k}$ and the
simple zero-eigenvalue corresponding to the generator of the $xy$-rotations $%
-\bar{\mathcal{J}}a$. That is, 
\[
\ker D^{2}V(a;s_{k})=V_{k}\oplus V_{n-k}\oplus \bar{\mathcal{J}}{a}.
\]
In order to prove the existence of solutions $\nabla V(u;\mu )=0$
bifurcating from the trivial solution $u=a$ when the parameter $\mu $
crosses $s_{k}$, we consider the fixed point spaces of two subgroups 
\[
H_{1}=\mathbb{Z}_{2}\left( \tilde{\kappa}_{y}\right) ,\qquad H_{2}=\mathbb{Z}%
_{2}\left( \tilde{\kappa}_{y}\kappa _{z}\right) .
\]%
That is, let $\nabla V^{H_{j}}:Fix(H_{j})\rightarrow Fix(H_{j})$ be the
restriction of $\nabla V$ to the fixed point space of $H_{j}$ for $j=1,2$,
we will show that the restricted maps $\nabla V^{H_{j}}$ for $j=1,2$ have
the advantage that the zero-eigenvalue {$-\bar{\mathcal{J}}a$} is not
present in $D^{2}V^{H_{j}}$ and that the double zero-eigenvalue $-\mu +s_{k}$
becomes simple.

The fixed point space of $\mathbb{R}^{3n}$ under the action of $H_{1}$
satisfies the symmetries 
\begin{equation}
u_{0}=R_{y}u_{0},\qquad u_{n/2}=R_{y}u_{n/2},\qquad u_{j}=R_{y}u_{n-j},
\label{Sym}
\end{equation}%
and of $H_{2}$ 
\begin{equation}
u_{0}=R_{z}R_{y}u_{0},\qquad u_{n/2}=R_{z}R_{y}u_{n/2},\qquad
u_{j}=R_{z}R_{y}u_{n-j}\text{.}  \label{Sym2}
\end{equation}

\begin{theorem}
\label{Thm1}For each $k\in \lbrack 2,n/2]\cap \mathbb{N}$, there are two 
\emph{global bifurcations} of solutions of $\nabla V(u;\mu )=0$ from the
trivial solution $u={a}$ at $\mu =s_{k}$, one denoted by $\mathfrak{F}%
_{1}^{k}$ with symmetries (\ref{Sym}) and another denoted by $\mathfrak{F}%
_{2}^{k}$ with symmetries (\ref{Sym2}). Furthermore, the relative equilibria
in both families are formed by $n/h$-groups of regular $h$-polygons, where $h
$ is the greatest common divisor of $k$ and $n$.
\end{theorem}

\begin{proof}
We look for bifurcation of solutions of $\nabla V^{H_{1}}(u;\mu )=0$ from
the trivial solution $u=a$ at $\mu =s_{k}$. Using that $\ker
D^{2}V(a;s_{k})=V_{k}\oplus V_{n-k}\oplus \bar{\mathcal{J}}{a}$, it is easy
to see that $v^{n-k}\in V_{n-k}$ and $\bar{\mathcal{J}}{a}$ do not satisfy
the symmetry (\ref{Sym}), which implies that they do not belong to $%
Fix(H_{1})$. Thus the kernel of $D^{2}V^{H_{1}}({a};s_{k})$ is one
dimensional and generated by $v^{k}\in V_{k}$. The eigenvalue $-\mu +s_{k}$
corresponding to the eigenvector $v^{k}\in V_{k}$ crosses zero at $\mu
=s_{k} $. Using Brouwer degree as in section \textquotedblleft 3.
Bifurcation theorem\textquotedblright\ in \cite{MR2832692}, we can prove the
existence of a \emph{global bifurcation} of solutions of $\nabla
V^{H_{1}}(u;\mu )=0$ from the trivial solution $u={a}$ at $\mu =s_{k}$.
Furthermore, since $\tilde{\zeta}^{n/h}$ leaves the subspace $V_{k}$ fixed,
because $\tilde{\zeta}^{n/h}z=\left( e^{ik2\pi /n}\right) ^{n/h}z=z$
according to (\ref{Ac}), then the family of solutions $\mathfrak{F}_{1}^{k}$%
\ arising from $\mu =s_{k} $\ is fixed by the group $\mathbb{\tilde{Z}}_{h}$
generated by $\tilde{\zeta}^{n/h}$. This implies that the solution is formed
by $n/h$-polygons (see \cite{MR2832692} for details). The proof in the case $%
k=n/2$ and $H_{2}$ is analogous, the only key difference is that 
\begin{equation*}
\ker D^{2}V^{H_{2}}(a;s_{k})=\ker D^{2}V(a;s_{k})\cap Fix(H_{2})=V_{n-k},
\end{equation*}%
because $v^{k}\in V_{k}$ and $\bar{\mathcal{J}}{a}$ do not satisfy the
symmetry (\ref{Sym2}), and they do not belong to $Fix(H_{2})$.
\end{proof}

The polygon $a$ is a relative equilibrium only when $\omega =\mu -s_{1}>0$,
which requires that $\mu >s_{1}$. In a slightly different context it was
noticed by R. Moeckel \cite{Moeckel_1990} that the condition $\mu =n>s_{1}$
holds only for $n<473$. An interesting consequence of this fact is that for
a non-ionized atom the $n$-polygon is a relative equilibrium only for an
atomic number less than $473$. We obtain numerically the additional inequalities $%
s_{1}<n<s_{2}$ for $n=3$ and $n\geq 12$, $s_{2}<n<s_{3}$ for $%
n=4,5,8,9,10,11$ and $s_{3}<n<s_{4}$ for $n=6,7$. Since the bifurcations of
relative equilibria arising from the polygons at $\mu =s_{k}$ are
subcritical, for instance, one can deduce from these inequalities that it is
unlikely to find a relative equilibrium with $\mu =n$ in the bifurcation
from $s_{1}$ for the cases $n=4,...,11$.
\subsection{Periodic solutions arising from spatial relative equilibria}

Now we turn the attention to the analysis of non-trivial $2\pi /\nu $%
-periodic solutions of (\ref{Equ}) arising from a spatial relative
equilibrium $(u_{0};\mu _{0})$, which for the present paper belongs to the
families $\mathfrak{F}_{1}^{k}$ or $\mathfrak{F}_{2}^{k}$. For the
validation of the hypotheses necessary to obtain the periodic solutions we
will use the computer-assisted proof technique of Section 3. These
hypotheses are easier to verify in the equivalent system 
\begin{equation}  \label{Eqs}
\begin{pmatrix}
\dot{u} \\ 
\dot{v}%
\end{pmatrix}
= 
\begin{pmatrix}
v \\ 
-2\sqrt{\mu -s_{1}}\mathcal{\bar{J}}v+\nabla V(u;\mu )%
\end{pmatrix}%
.
\end{equation}

The linearization of equation (\ref{Eqs}) at a spatial relative equilibrium $%
(u_{0};\mu _{0})$ is 
\[
\binom{\dot{u}}{\dot{v}}=L(u_{0};\mu _{0})\binom{u}{v},\qquad L(u_{0};\mu
_{0}) \,\overset{\mbox{\tiny\textnormal{\raisebox{0ex}[0ex][0ex]{def}}}}{=}%
\, \left( 
\begin{array}{cc}
0 & I \\ 
D ^{2}V(u_{0};\mu ) & -2\sqrt{\mu _{0}-s_{1}}\mathcal{\bar{J}}%
\end{array}%
\right) .
\]
Notice that $\lambda $ is an eigenvalue of $L(u_{0};\mu _{0})$ with
eigenvector $(u,v)$ if and only if $v=\lambda u$ and%
\[
-2\sqrt{\mu _{0}-s_{1}}\mathcal{\bar{J}}v+D^{2}V(u_{0};\mu _{0})u=\lambda v.
\]%
This condition is equivalent to $u\in \ker \tilde{M}(\lambda )$, where 
\[
\tilde{M}(\lambda )=-\lambda ^{2}I-2\sqrt{\mu _{0}-s_{1}}\lambda \mathcal{%
\bar{J}}+D^{2}V(u_{0};\mu _{0}).
\]

\begin{remark}
\label{rem:eigs_relations} Since $\tilde{M}(\lambda )=\tilde{M}(-\lambda
)^{T}$, then $\det \tilde{M}(\lambda )=\det \tilde{M}(-\lambda )$ is an even
real polynomial in $\lambda $. Thus $\bar{\lambda},-\lambda ,-\bar{\lambda}$
are eigenvalues of $L(u_{0};\mu _{0})$ if $\lambda \in \mathbb{C}$ is an
eigenvalue of $L(u_{0};\mu _{0})$. Actually, this is an immediate
consequence of the fact that the matrix $L(u_{0};\mu _{0})$ is a
reformulation, as first order system for positions and velocities, of a
Hamiltonian matrix \cite{MR1140006}.
\end{remark}

The purely imaginary eigenvalues $\lambda =i\nu _{0}$ of $L(u_{0};\mu _{0})$
give the (normal) frequencies of the periodic solutions of the linearized
system. The periodic solutions of the linearized system persist in the
nonlinear system (non-linear normal modes) under the assumptions of the
Lyapunov center theorem \cite{MR1140006}. The main assumption is that $i\nu
_{0}$ is non-resonant, which means that the eigenvalue $i\nu _{0}$ is a
simple eigenvalue of $L(u_{0};\mu_{0}) $ and $il\nu _{0}$ is not an
eigenvalue of $L(u_{0};\mu_{0})$ for any integer $l\neq 1$.

The classical Lyapunov center theorem cannot be applied directly because the
equilibrium $a$ is not isolated and zero is always an eigenvalue of $%
L(u_{0};\mu _{0})$ due to the $SO(2)$-action. Other equivariant versions of
the Lyapunov theorem consider these circumstances such as \cite{MR1984999}
and \cite{MR2003792}. In order to use a simple version of those results, we
make the following definition.

\begin{definition}
\label{Def1} We say that $i\nu _{0}$ is a $SO(2)$-nonresonant eigenvalue of $%
L(u_{0};\mu _{0}) $ if $i\nu _{0}$ is a simple eigenvalue of $L(u_{0};\mu
_{0}) $, $0$ is a double eigenvalue of $L(u_{0};\mu _{0})$ due to the action
of the group $SO(2)$, and $il\nu _{0}$ is not an eigenvalue of $L(u_{0};\mu
_{0}) $ for integers $l\geq 2$.
\end{definition}

In the case that $i\nu _{0}$ is a $SO(2)$-nonresonant eigenvalue of $L(u_{0};\mu _{0})$ with
eigenvector $(u,v)$, the first order asymptotic expansion of the family of periodic solutions is given by
\[
u(t)=u_{0}+\varepsilon \operatorname{Re}(e^{i\nu _{0}t}u)+O(\varepsilon^2).
\]
We use this fact to produce the illustrations of the periodic solutions in Figures \ref{fig:example1} and \ref{fig:example2}.
\begin{theorem}
\label{Thm2} A relative equilibrium $(u_{0};\mu _{0}) $ has a global family
of $2\pi /\nu $-periodic solutions arising from $u_{0}$ with initial
frequency $\nu =\nu _{0}$ when $L(u_{0};\mu _{0}) $ has a $SO(2)$%
-nonresonant eigenvalue $i\nu _{0}$.
\end{theorem}

\begin{proof}
Looking for $2\pi /\nu $-periodic solutions of the equation $\ddot{u}+2\sqrt{%
\mu -s_{1}}\mathcal{\bar{J}}\dot{u}=\nabla V(u;\mu )$ is equivalent to look
for zeros $x(t)=u(t\nu )$ of the map%
\begin{equation*}
\mathcal{F}(x;\mu ,\nu )=-\nu ^{2}\ddot{x}-2\sqrt{\mu -s_{1}}\mathcal{\bar{J}%
}\nu \dot{x}+\nabla V(u;\mu ):H_{2\pi }^{2}\times \mathbb{R}^{2}\rightarrow
L_{2\pi }^{2}\text{.}
\end{equation*}%
We consider that $\mu =\mu _{0}$ is fixed, then $u_{0}$ satisfies $\mathcal{F%
}(u_{0};\mu _{0},\nu )=\nabla V(u_{0};\mu _{0})=0$. The linearization $D%
\mathcal{F}(u_{0};\mu _{0},\nu )$ in Fourier components $x=\sum x_{l}e^{ilt}$
is $D\mathcal{F}(u_{0};\mu _{0},\nu )=\sum_{l\in \mathbb{Z}}M(l\nu
)x_{l}e^{ilt}$, where 
\begin{equation*}
M(\nu )=\tilde{M}(i\nu )=\nu ^{2}I-2\sqrt{\mu _{0}-s_{1}}i\nu \mathcal{\bar{J%
}}+D^{2}V(u_{0};\mu _{0})
\end{equation*}%
is a self-adjoint matrix. Thus, the assumption that $i\nu _{0}$ is $SO(2)$%
-nonresonant implies that $\ker M(0)$ is generated by $\bar{\mathcal{J}}{a}$%
, $\nu _{0}$ is a simple zero of $\det M(\nu _{0})$ and $M(l\nu _{0})$ is
invertible for integers $l\geq 2$.

Therefore, the kernel of the linearized operator $D\mathcal{F}(u_{0};\mu
_{0},\nu _{0})$ consists exactly of $\bar{\mathcal{J}}{a}$ and the real and
imaginary parts of $e^{it}w$ with $w\in \ker M(\nu _{0})$. These are the
necessary hypotheses in order to prove the bifurcation theorem in Section 6
in \cite{MR3007103}. We proceed analogously: Let $\sigma $ be the sign of
the determinant of $D^{2}V(u_{0};\mu _{0})$ in the orthogonal complement to $%
\bar{\mathcal{J}}{a}$ (the generator of the $SO(2)$-orbit). Under the
non-resonance assumption, we have that $\sigma \neq 0$. Let $n(\nu )$ be the
Morse index of the self-adjoint matrix $M(\nu )$. Since $M(l\nu _{0})$ is
invertible for integers $l\geq 2$, according to Section 6 in \cite{MR3007103}%
, we have that the $SO(2)\times S^{1}$-equivariant index of $\mathcal{F}%
(x;\mu_{0},\nu )$ at the orbit of $u_{0}$ is $\sigma n(\nu )(\mathbb{Z}_{1})$%
. In \cite{MR1984999} is proven that a global bifurcation of periodic
solutions exists if this index changes. The result follows from the fact
that $\nu _{0}$ is a simple root of the polynomial $\det M(\nu )=0$, i.e.
the Morse index $n(\nu )$ necessarily changes at $\nu _{0}$.
\end{proof}

\begin{remark}
The local existence of the family of periodic solutions can be proven using
Poincar\'{e} sections as in \cite{MR2003792}. It is also possible to use
equivariant degree theory to prove the existence of the family of periodic
solutions even for $SO(2)$-resonant eigenvalues, but for validating the
hypotheses with computer-assisted proofs it is simple to consider the case
of $SO(2)$-nonresonant eigenvalues.
\end{remark}

\section{Computer-assisted proofs of relative equilibria}

In this section, we use a pseudo-arclength continuation method (e.g. see \cite{MR910499}) to numerically compute branches of steady states in the families $\mathfrak{F}_{1}^{k}$ and $\mathfrak{F}_{2}^{k}$. Along with the numerical continuation, we obtain computer-assisted proofs of the equilibria and we verify the hypotheses of Theorem~\ref{Thm2} to conclude the existence of families of periodic orbits. 
%Note that we do apply the same program to the family $\mathfrak{F}_{2}^{k}$. The reason is that the second family has an extra pair of eigenvalues close to zero, and complicates the rigorous validation of the eigenvalues. Since their validation requires additional considerations out of the scope of the present work, and in order to keep the present treatment simple, we treat only the first family.
This is done using a Newton-Kantorovich type theorem, which is similar to the Krawczyk operator's approach \cite{MR657002,MR255046} and the interval Newton method \cite{MR0231516}. The presented formulation is inspired by the so-called {\em radii polynomial approach} (e.g. see \cite{MR2338393,HLM,MR3917433}), which is also a variant of the Newton-Kantorovich Theorem.

Consider a finite dimensional Banach space $X$ (in our context $X=\R^N$ or $X=\C^N$, for some $N \in \N$).
Choose a norm $\| \cdot\|_X$ on $X$. Given a point $y \in X$ and a radius $r>0$, denote by $B_r(y) = \{ x \in X : \|y - x\| < r \}$ the open ball of radius $r$ centered at $y$. Similarly, denote by $\overline{B_r(y)}$ the closed ball.

\begin{theorem} \label{thm:newton-kantorovich}
Let $U \subset X$ be an open set. Consider a Fr\'echet differentiable mapping $F: U \to X$ and fix a point $\bu \in U$ (an approximate zero of $F$). Let $A$ be an approximate inverse of the Jacobian matrix $DF(\bu)$ (that is $\| I - A DF(\bu) \|_{B(X)} \ll 1$), where $I$ is the identity on $X$ and where $\| \cdot \|_{B(X)}$ denotes the operator/matrix norm induced by the norm $\| \cdot \|_X$ on $X$. 
%Assume that the derivative $DF(\bu)$ is invertible. 
Fix $r_*>0$.
Suppose that the bounds $Y, Z=Z(r_*) > 0$ satisfy
\[
\| A F(\bu)\|_X \leq Y 
\quad \text{and} \quad
\sup_{z \in \overline{B_{r_*}(\bu)}} \| I - A DF(z) \|_{B(X)} \leq Z.
\]
Define
\begin{equation} \label{eq:radii_polynomial}
p(r) = (Z-1)r + Y.
\end{equation}
If there exists $r_0 \in (0, r_*]$ such that $p(r_0) < 0$, then there is a unique $\tu \in B_{r_0}(\bu) $ such that $F(\tu) = 0$.
\end{theorem}

\begin{proof}
Define the Newton-like operator $T:U \to X$ by 
\[
T(u) = u - A F(u),
\]
and note that $DT(u) = I - A  DF(u)$. The idea of the proof is to show that $T: \overline{B_{r_0}(\bu)} \to B_{r_0}(\bu)$ is a contraction. Consider $r_0 \in (0, r_*]$ such that $p(r_0) < 0$. Then $Zr_0 + Y < r_0$, and since $ r_0 $ is not zero, we have that
$Z \le Z + \frac{Y}{r_0} < 1$. For $x,y \in \overline{B_{r_0}(\bu)}$ we use the Mean Value Inequality to get that
\begin{align*}
\|T(u) - T(v)\|_X &\le \sup_{z \in \overline{B_{r_0}(\bx)}} \|DT(z)\|_{B(X)} \|u - v\|_X 
\\
& = \sup_{z \in \overline{B_{r_0}(\bu)}} \|I - A DF(z)\|_{B(X)} \|u - v\|_X
\\
&\le \sup_{z \in \overline{B_{r_*}(\bu)}} \| I - ADF(z)\|_{B(X)}  \|u - v\|_X
\\
& \le Z \|u - v\|_X.
\end{align*}
Since $ Z < 1 $, $T$ is a contraction on $\overline{B_{r_0}(\bu)}$. To see that $T$ maps the closed ball into itself (in fact in the open ball) choose $u \in \overline{B_{r_0}(\bu)}$, and observe that
\begin{align*}
\|T(u) - \bu \|_X &\leq \| T(u) - T(\bu)\|_X + \|T(\bu) - \bu \|_X \\
&\leq Z \|u - \bu\|_X + \| A F(\bu)\|_X \\
&\leq Zr_0 + Y = p(r_0)+r_0< r_0,
\end{align*}
which shows that $T(u) \in B_{r_0}(\bu)$ for all $u \in \overline{B_{r_0}(\bu)}$.
It follows from the contraction mapping theorem that there exists a unique $\tu \in \overline{B_{r_0}(\bu)}$ such that $T(\tu)=\tu \in B_{r_0}(\bu)$. Since $Z<1$, we get 
\[
 \| I - A DF( \bu) \|_{B(X)} \le 
\sup_{z \in \overline{B_{r_*}(\bu)}} \| I - A DF(z) \|_{B(X)} \leq Z <1,
\]
and hence $A DF( \bu)$ is invertible. From this we get that $A$ is invertible. 
By invertibility of $A$ and by definition of $T$, the fixed points of $T$ are in one-to-one correspondence with the the zeros of $F$.
We conclude that there is a unique $\tu \in B_{r_0}(\bu)$ such that $F(\tu)=0$.
\end{proof}

In practice, we perform the rigorous computation of the bounds $Y$ and $Z$ with interval arithmetic \cite{MR0231516} in MATLAB using the library INTLAB \cite{Ru99a}.

\subsection{The first family of spatial relative equilibria} \label{sec:CAP_rel_eq}

To proceed with the computer-assisted proof of the relative equilibria we need to find an explicit
representation for $\nabla V^{H_{1}}:Fix(H_{1})\rightarrow Fix(H_{1})$. For
this purpose, we define the subspace%
\begin{equation*}
X_{1}=\left\{ \tilde{u}=(\tilde{u}_{0},\tilde{u}_{1},....,\tilde{u}%
_{[n/2]})\in \mathbb{R}^{3}\times ...\times \mathbb{R}^{3}:\tu%
_{j}=(x_{j},0,z_{j})\in \mathbb{R}^{3},~j=0,n/2\right\} 
\end{equation*}%
and the isomorphism%
\begin{equation*}
\iota _{1}:X_{1}\rightarrow Fix(H_{1}),\qquad \iota _{1}(\tilde{u})=\left(
u_{0},u_{1},....,u_{n}\right) ,
\end{equation*}%
given by $u_{j}=\tilde{u}_{j}$ for $j\in \lbrack 0,n/2]\cap \mathbb{N}$ and $%
u_{j}=R_{y}\tilde{u}_{n-j}$ for $j\in (n/2,n-1]\cap \mathbb{N}$. Therefore, the zeros of $\nabla V^{H_{1}}:Fix(H_{1})\rightarrow Fix(H_{1})$ correspond
to the zeros of%
\begin{equation*}
F_{1}\bydef\iota _{1}^{-1}\circ \nabla V^{H_{1}}\circ \iota _{1}:X_{1}\rightarrow
X_{1}.
\end{equation*}
More explicitly, we have that
\begin{equation}
F_{1}=(f_{0},f_{1},....,f_{[n/2]}):X_{1}\rightarrow X_{1},
\label{eq:eq_family_one}
\end{equation}
where 
\begin{equation*}
f_{j}(\tilde{u};\mu )=\left( \mu -s_{1}\right) \bar{I}u_{j}-\mu \frac{u_{j}}{%
\left\Vert u_{j}\right\Vert ^{3}}+\sum_{0\leq i\leq n/2(i\neq j)}\frac{%
u_{j}-u_{i}}{\left\Vert u_{j}-u_{i}\right\Vert ^{3}}+\sum_{0<i<n/2}\frac{%
u_{j}-R_{y}u_{i}}{\left\Vert u_{j}-R_{y}u_{i}\right\Vert ^{3}}.
\end{equation*}%
This fact can be verified directly. We conclude that the families of solutions of $F_{1}(\tilde{u};\mu )=0$
are the critical solutions of $V(u;\mu )$ with $u=\iota _{1}\left( \tilde{u}%
\right) \in Fix(H_{1})$.

To compute numerically the family $\mathfrak{F}_{1}^{k}$ (that is solutions
of $F_{1}(\tilde{u};\mu )=0$), we apply the pseudo-arclength continuation
method \cite{MR910499}, which we now briefly review. Denote $N \bydef 3([n/2]+1)-2$, so that $X_{1} \cong \R^N$. Using that notation, $F_{1}:\R^{N+1} \to \R^N$. The idea of the pseudo-arclength continuation is to treat the parameter $\mu$ as a variable, to set $U \bydef (\tilde{u};\mu) \in \R^{N+1}$ and perform a continuation with respect to the {\em pseudo-arclength parameter}. The process begins with a solution $U_0$ (exact or numerical given within a prescribed tolerance). To produce a {\em predictor}, which will serve as an initial condition to Newton's method, we compute a tangent vector $\dot U_0$ (of unit length) to the curve at $U_0$. It can be computed using the formula
\[
D_U F_1(U_0) \dot U_0 = \left[ D_{\tilde{u}} F_1(U_0) ~~  \frac{\partial F_1}{\partial \mu}(U_0) \right] \dot U_0 = 0 \in \R^N.
\]
Denoting the pseudo-arclength parameter by $\Delta_s>0$, set the predictor to be
\[
\hat{U}_1 \bydef U_0 + \Delta_s \dot U_0 \in \R^{N+1}.
\]
The {\em corrector} step then consists of converging back to the solution curve on the hyperplane perpendicular to the tangent vector $\dot U_0$ which contains the predictor $\hat{U}_1$. The equation of this plan is given by
$E(U) \bydef (U-\hat U_1) \cdot \dot U_0 = 0$. Then, we apply Newton's method to the new function
\begin{equation} \label{eq:PAL_fonction}
U \mapsto \begin{pmatrix} E(U) \\ F_1(U) \end{pmatrix}
\end{equation}
with the initial condition $\hat U_1$ in order to obtain a new solution $U_1$ given again within a prescribed tolerance. We reset $U_1 \mapsto U_0$ and start over. At each step of the algorithm, the function defined in \eqref{eq:PAL_fonction} changes since the plane $E(U)=0$ changes. With this method, it is possible to continue past folds. Repeating this procedure iteratively produces a branch of solutions.

We initiate the numerical continuation from the trivial solutions $\tilde{u}=\iota_{1}^{-1}(a)$ at $\mu =s_{k}$. More explicitly, at the beginning, we set $U_0 = (\iota_{1}^{-1}(a),s_{k}) \in \R^{N+1}$. Then, along the continuation, we use Theorem~\ref{thm:newton-kantorovich} to verify the existence (with tight rigorous
error bounds) of several solutions of $F_1=0$ (with $F_{1}$ defined in \eqref{eq:eq_family_one}), hence yielding spatial relative equilibria in the family $\mathfrak{F}_{1}^{k}$. See Figure~\ref{fig:all_plots} for plots of several continuations.

A similar analysis and numerical implementation have been implemented for the family of solutions $\mathfrak{F}_{2}^{k}$ satisfying the symmetry (\ref{Sym2}), by using instead the map
\begin{align*}
F_{2}(\tilde{u};\mu )&=(f_{0},f_{1},....,f_{[n/2]}):X_{2}\rightarrow X_{2}, \\
X_{2}&=\left\{ \tilde{u}=(u_{0},u_{1},....,u_{[n/2]}):u_{j}=(x_{j},0,0)\in \mathbb{R}^{3},\quad
j=0,n/2\right\} ,
\end{align*}%
where%
\begin{equation*}
f_{j}(\tilde{u};\mu )=\left( \mu -s_{1}\right) \bar{I}u_{j}-\mu \frac{u_{j}}{%
\left\Vert u_{j}\right\Vert ^{3}}+\sum_{0\leq i\leq n/2(i\neq j)}\frac{%
u_{j}-u_{i}}{\left\Vert u_{j}-u_{i}\right\Vert ^{3}}+\sum_{0<i<n/2}\frac{%
u_{j}-R_{z}R_{y}u_{i}}{\left\Vert u_{j}-R_{z}R_{y}u_{i}\right\Vert ^{3}}.
\end{equation*}%
See Figure~\ref{fig:all_plots2} for plots of several continuations in the family $\mathfrak{F}_{2}^{k}$.

\begin{remark}[\bf Colour coding for Figures~\ref{fig:all_plots} and~\ref{fig:all_plots2}]
The colour coding for the presentation of the relative equilibrium solutions in Figures~\ref{fig:all_plots} and~\ref{fig:all_plots2} is as follows. Each branch going from green to red represents the main branch, while each cyan to purple branch portrays a branch born from a secondary bifurcation from the main branch. The points in the following colours were not successfully validated with computer-assisted proofs for three reasons: (Blue) unable to verify the relative equilibria, (Black) unable to verify the eigenvalues and (Orange) unable to verify the nonresonance of the eigenvalues.
\label{coding}
\end{remark}

\begin{remark}
All the spatial relative equilibria are unstable near the polygon because the polygon is unstable according to the computations obtained in \cite{MR3007103}. Unfortunately, we
were not able to find linearly stable solutions from the numerical
exploration carried on for the branches, so all the relative
equilibria that we computed are unstable.
\end{remark}
\begin{figure}
 \vspace{-1.5cm}
     \centering
     \begin{subfigure}[b]{0.49\textwidth}
         \centering
         \includegraphics[width=\textwidth]{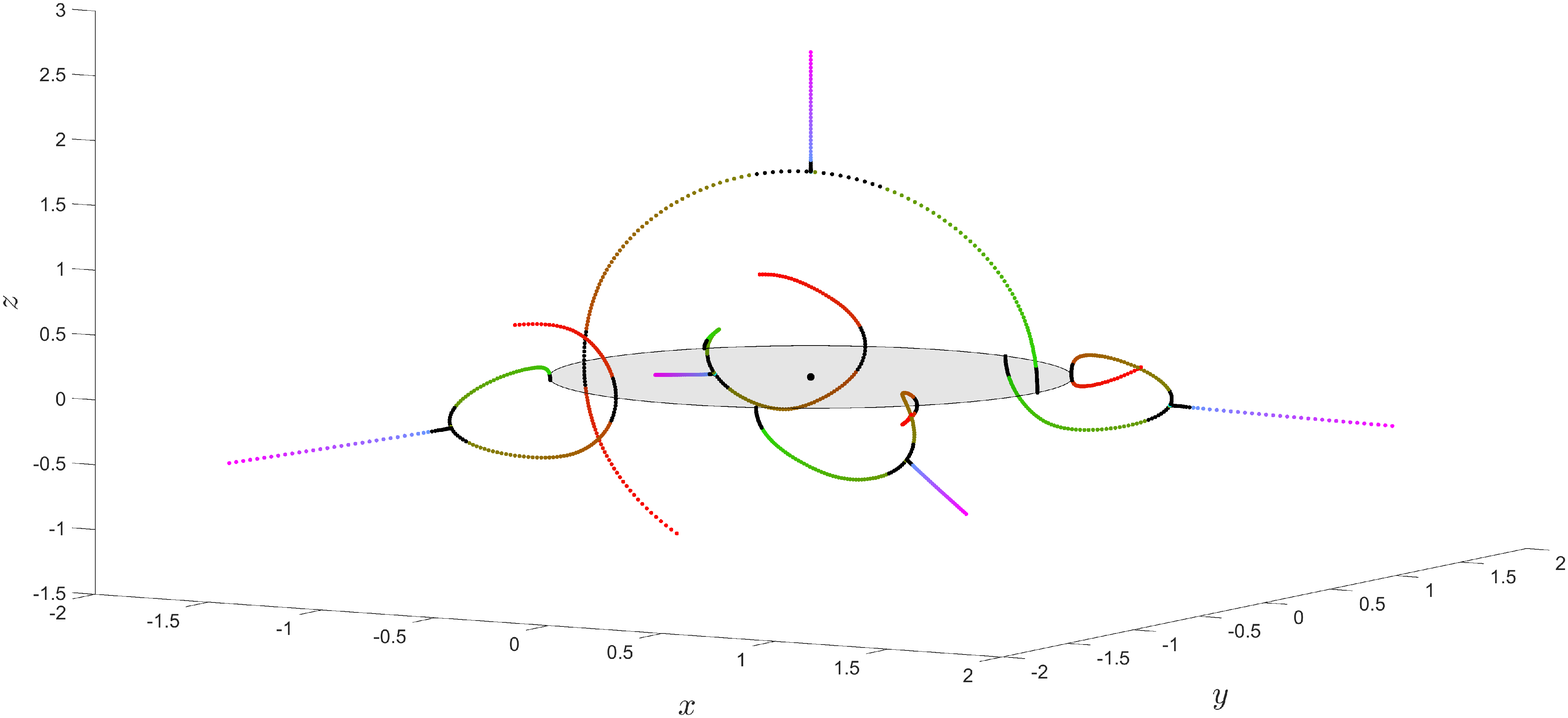}
         \caption{$n=5$, $k=2$}
         \label{fig:1a}
     \end{subfigure}
     \hfill
     \begin{subfigure}[b]{0.49\textwidth}
         \centering
         \includegraphics[width=\textwidth]{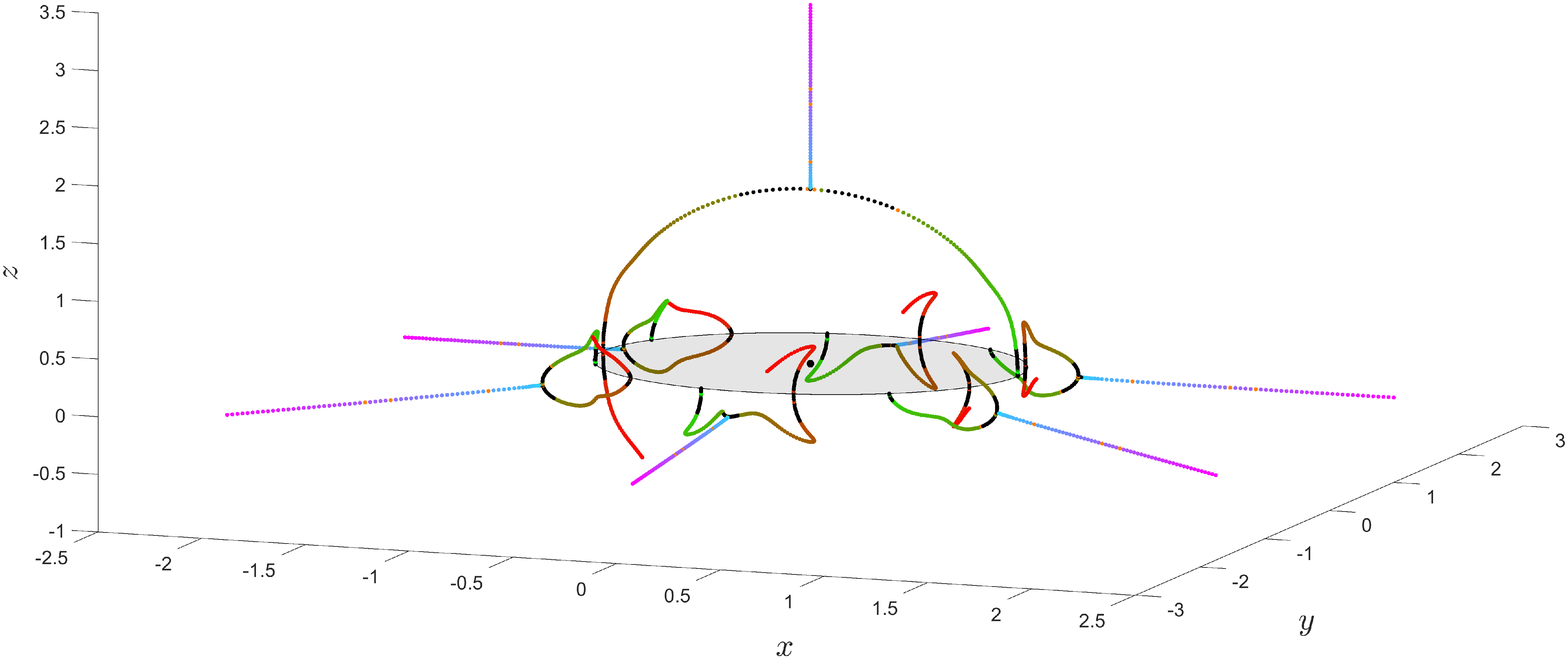}
         \caption{$n=7$, $k=2$}
         \label{fig:2a}
     \end{subfigure}
     \hfill
     \begin{subfigure}[b]{0.49\textwidth}
         \centering
         \includegraphics[width=\textwidth]{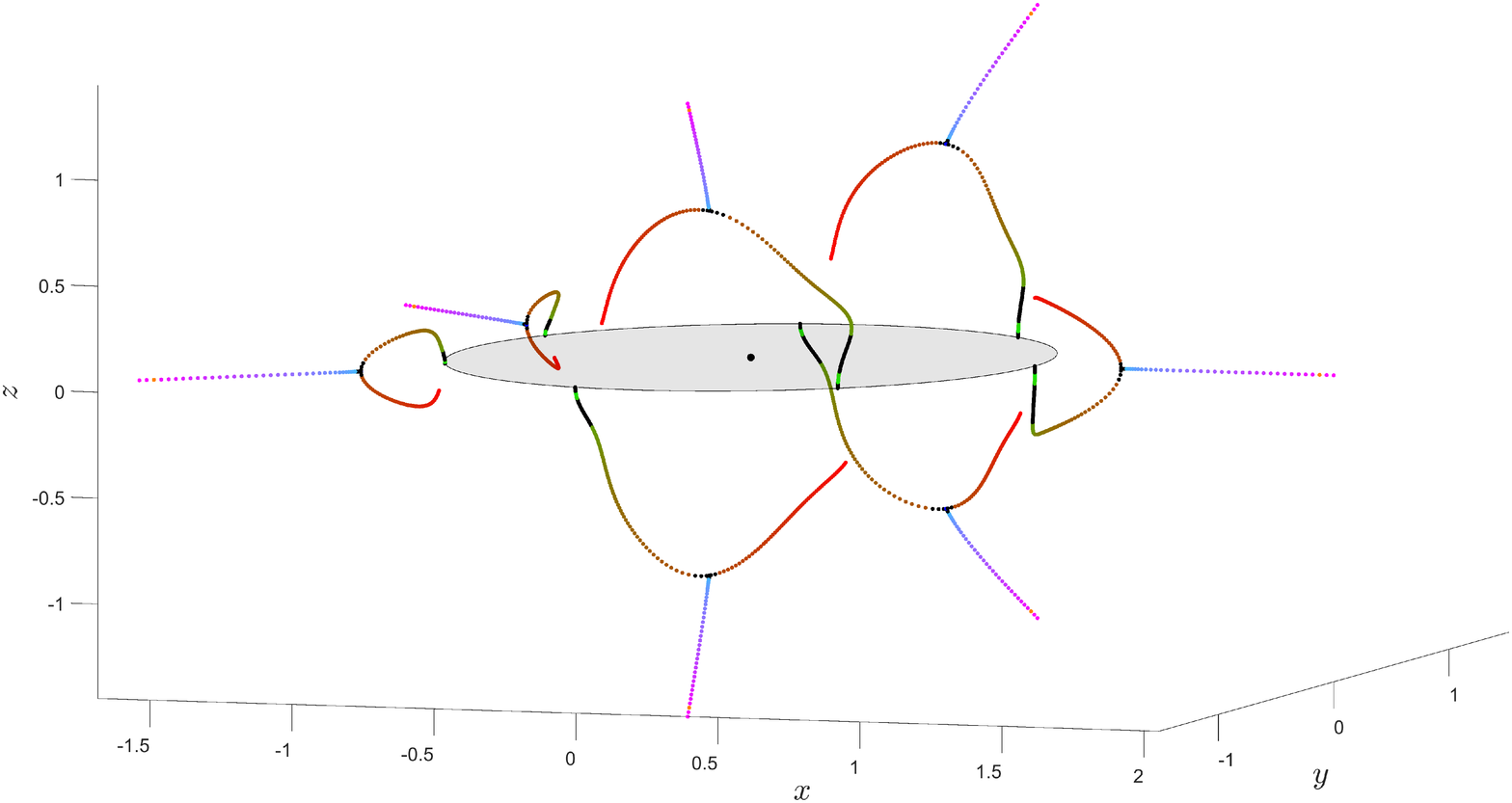}
         \caption{$n=7$, $k=3$}
         \label{fig:3a}
     \end{subfigure}
     \hfill
     \begin{subfigure}[b]{0.49\textwidth}
         \centering
         \includegraphics[width=\textwidth]{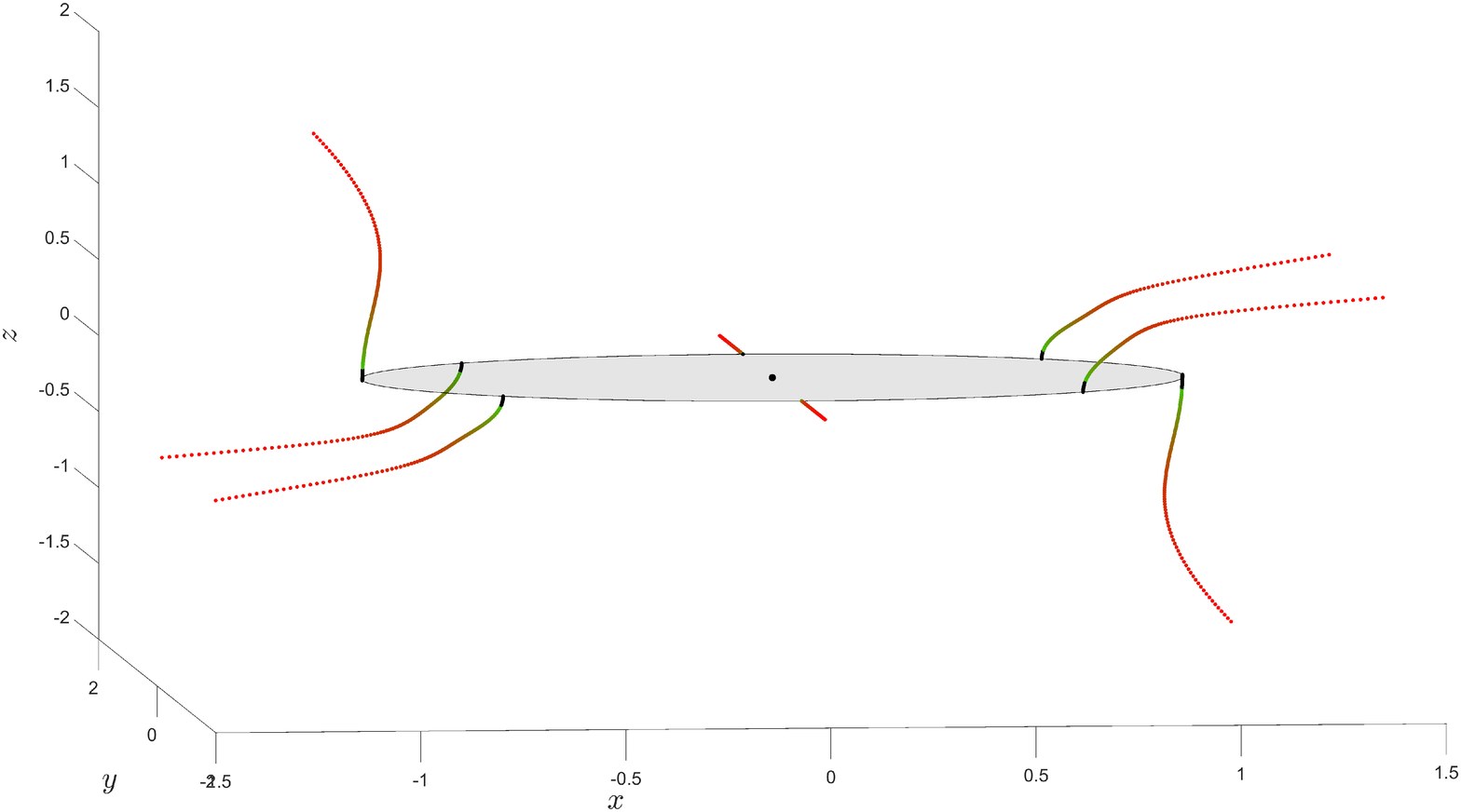}
         \caption{$n=8$, $k=3$}
         \label{fig:4a}
     \end{subfigure}
          \hfill
     \begin{subfigure}[b]{0.49\textwidth}
         \centering
         \includegraphics[width=\textwidth]{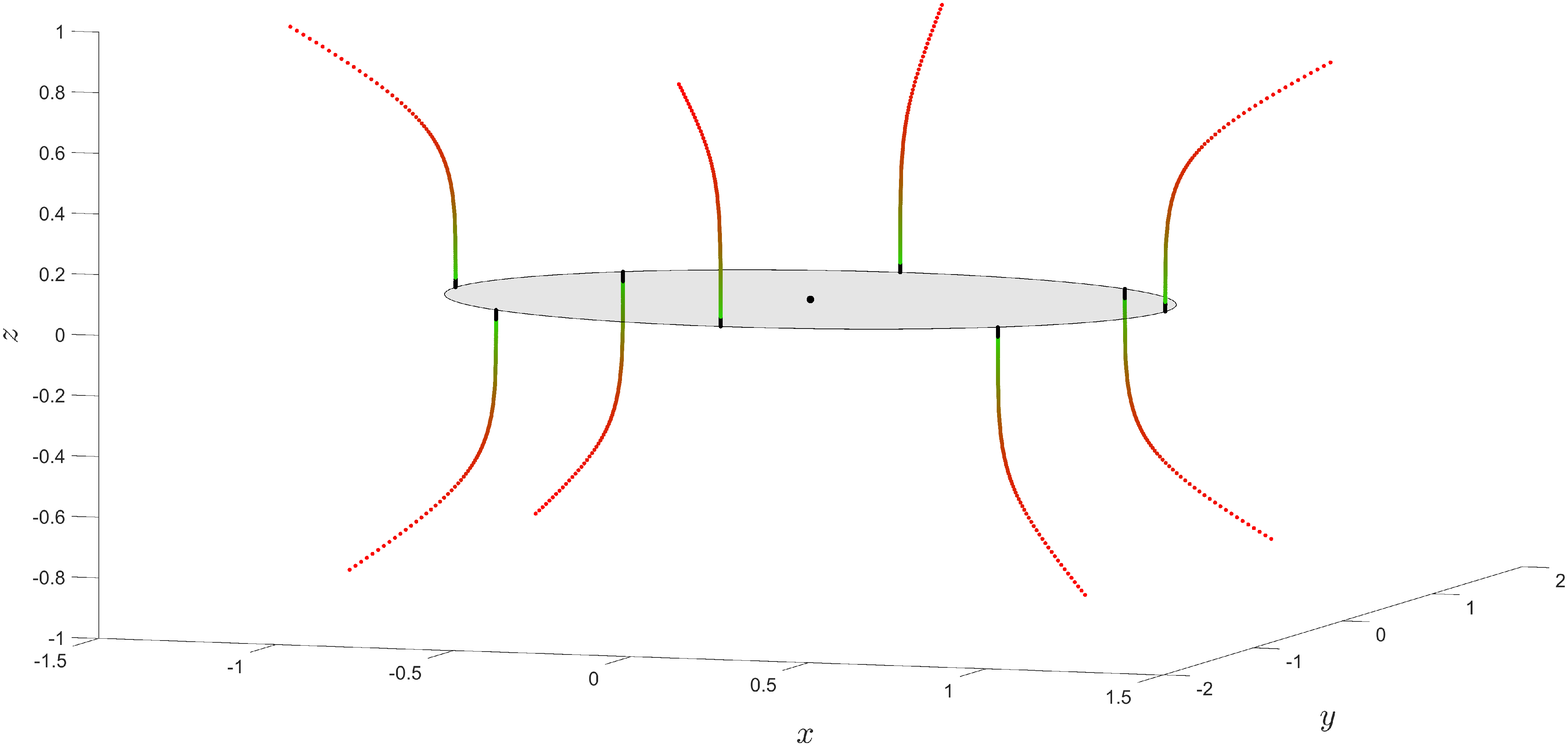}
         \caption{$n=8$, $k=4$}
         \label{fig:5a}
     \end{subfigure}
     \hfill
     \begin{subfigure}[b]{0.49\textwidth}
         \centering
         \includegraphics[width=\textwidth]{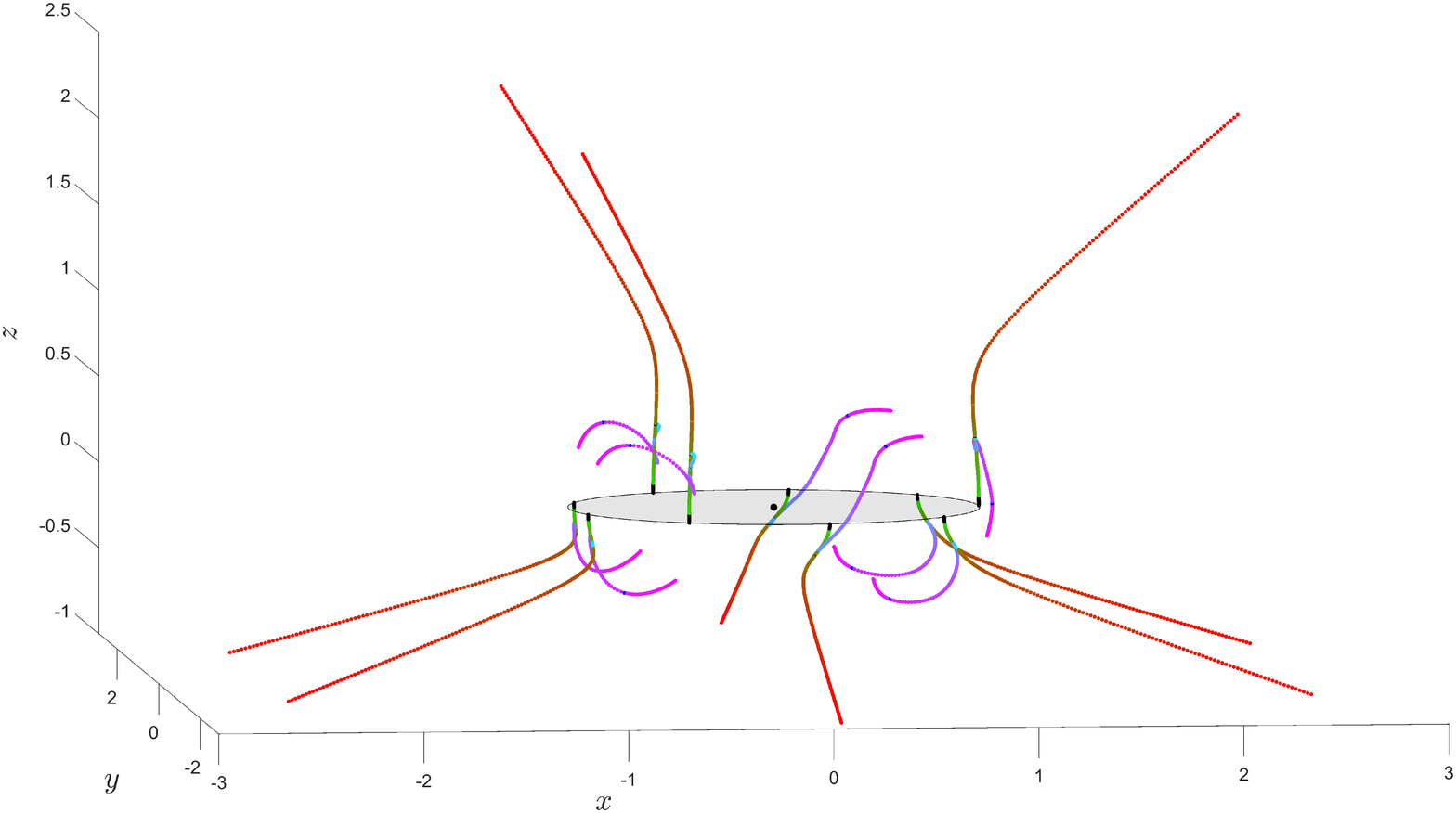}
         \caption{$n=9$, $k=3$}
         \label{fig:6a}
     \end{subfigure}
          \hfill
     \begin{subfigure}[b]{0.49\textwidth}
         \centering
         \includegraphics[width=\textwidth]{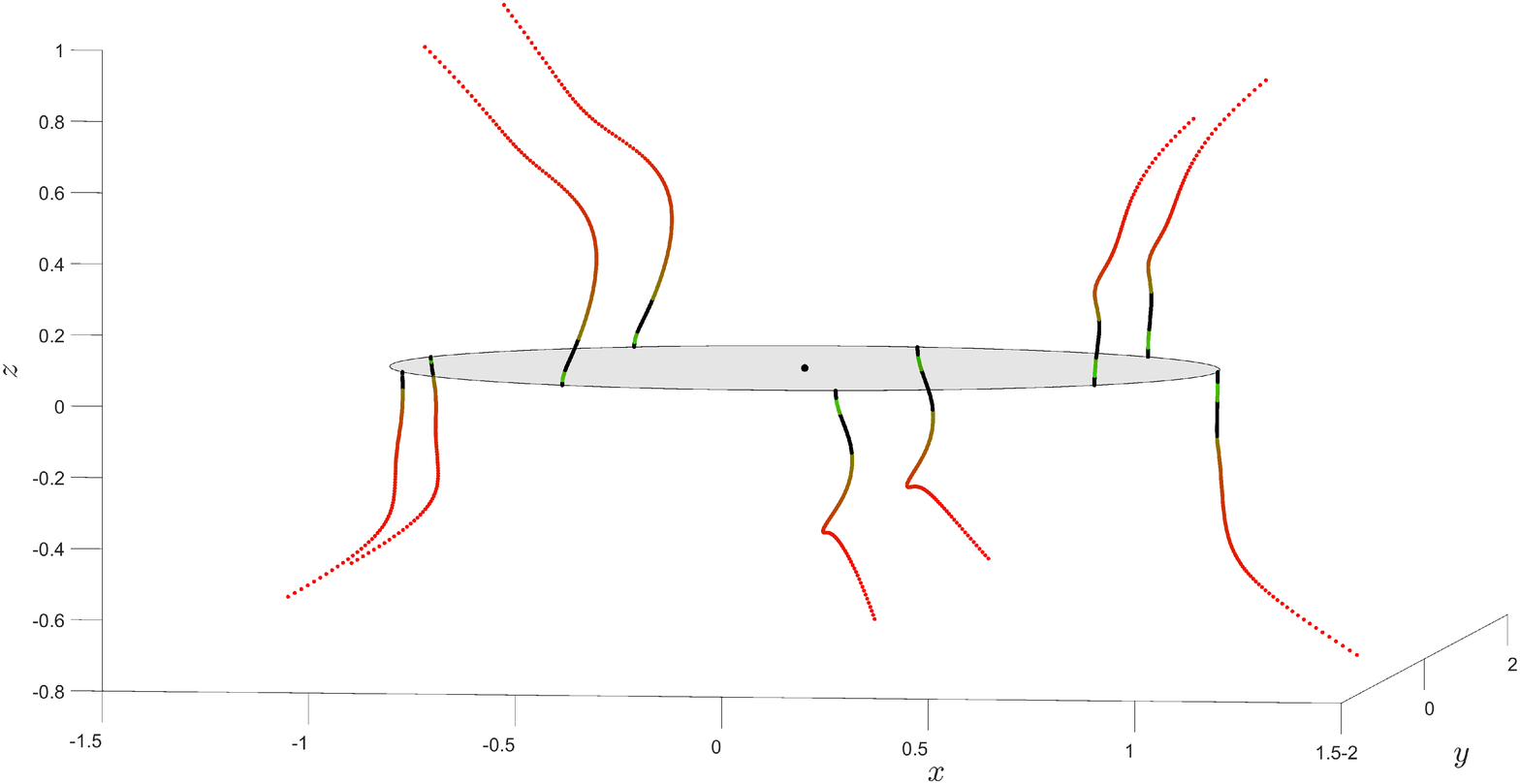}
         \caption{$n=9$, $k=4$}
         \label{fig:7a}
     \end{subfigure}
          \hfill
     \begin{subfigure}[b]{0.49\textwidth}
         \centering
         \includegraphics[width=\textwidth]{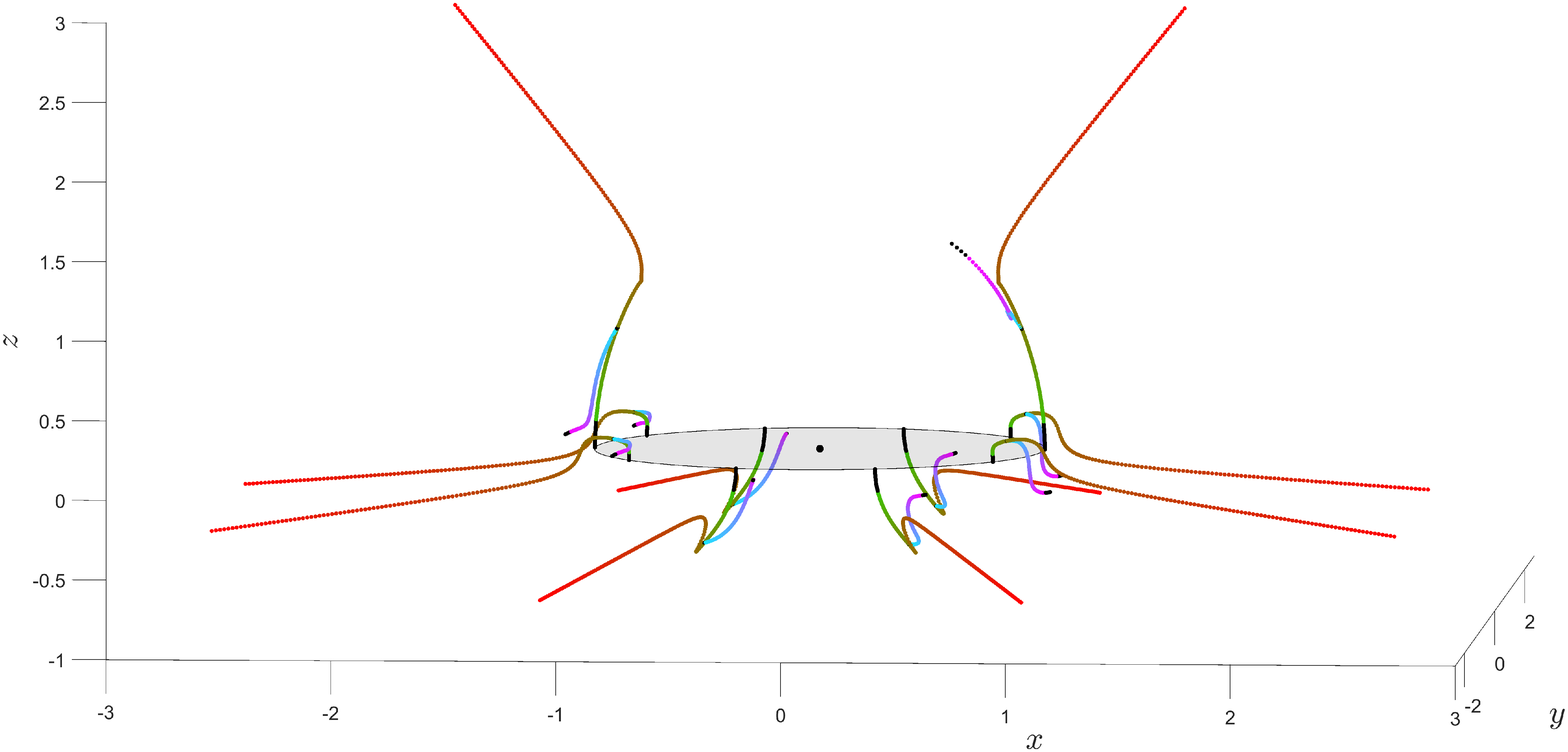}
         \caption{$n=10$, $k=2$}
         \label{fig:8a}
     \end{subfigure}
     \caption{Continuation of equilibria in the family $\mathfrak{F}_{1}^{k}$ for different values of $n$ and $k$. The colour coding in the figure is presented in Remark \ref{coding} }
        \label{fig:all_plots}
\end{figure}

\begin{figure}
     \centering
     \begin{subfigure}[b]{0.5\textwidth}
         \centering
         \includegraphics[width=\textwidth]{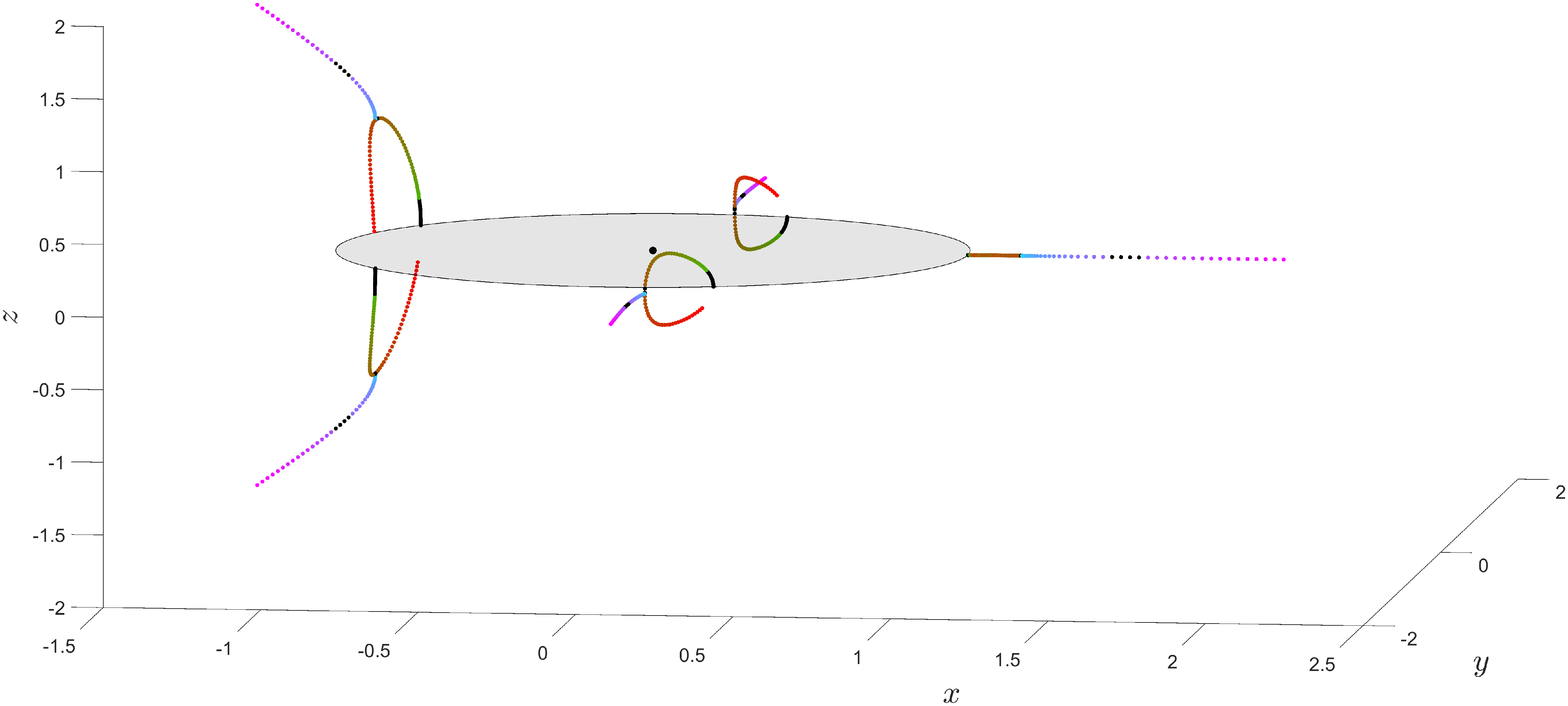}
         \caption{$n=5$, $k=2$}
         \label{fig:1b}
     \end{subfigure}
     \hfill
     \begin{subfigure}[b]{0.49\textwidth}
         \centering
         \includegraphics[width=\textwidth]{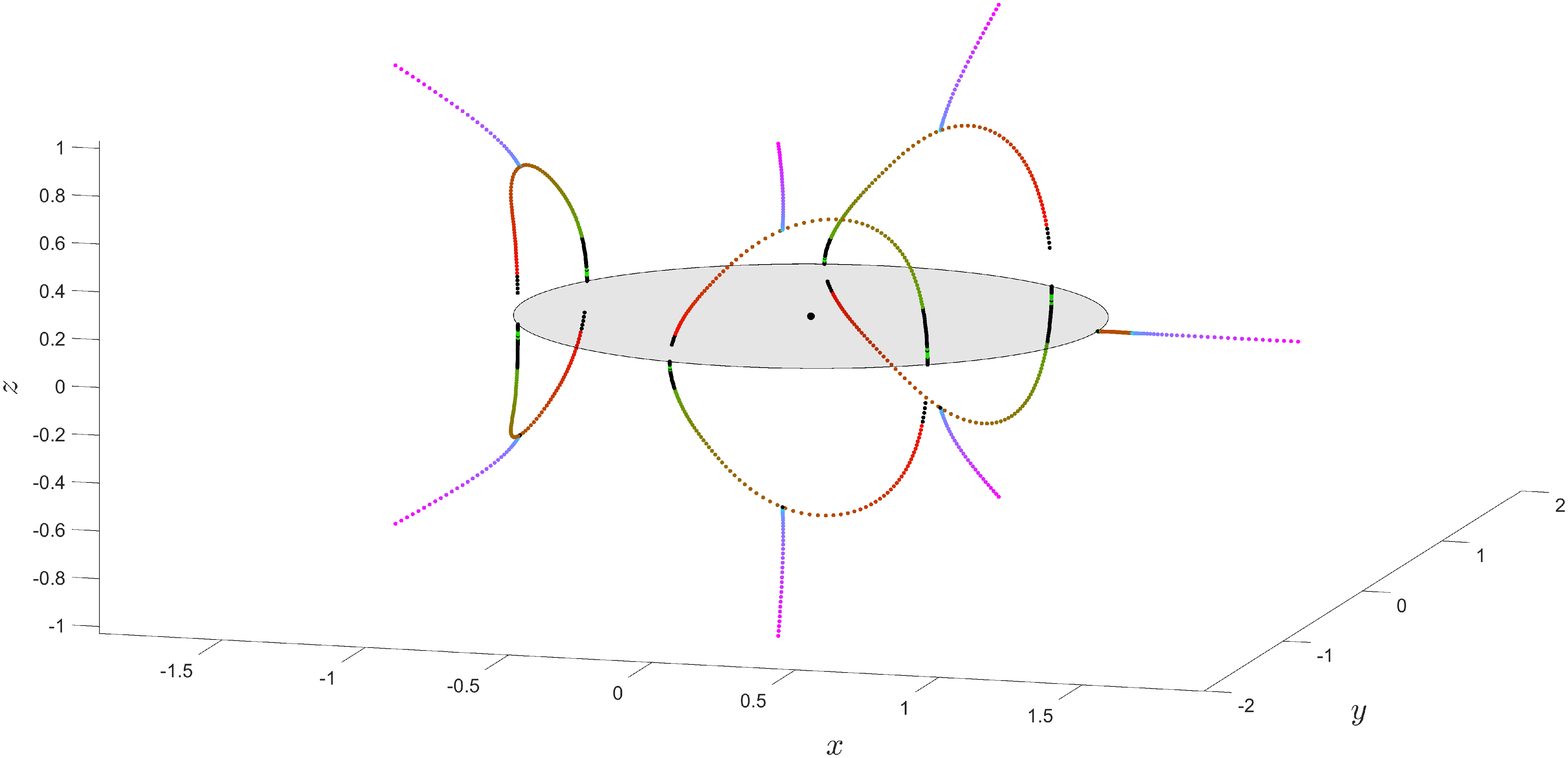}
         \caption{$n=7$, $k=2$}
         \label{fig:2b}
     \end{subfigure}
     \hfill
     \begin{subfigure}[b]{0.49\textwidth}
         \centering
         \includegraphics[width=\textwidth]{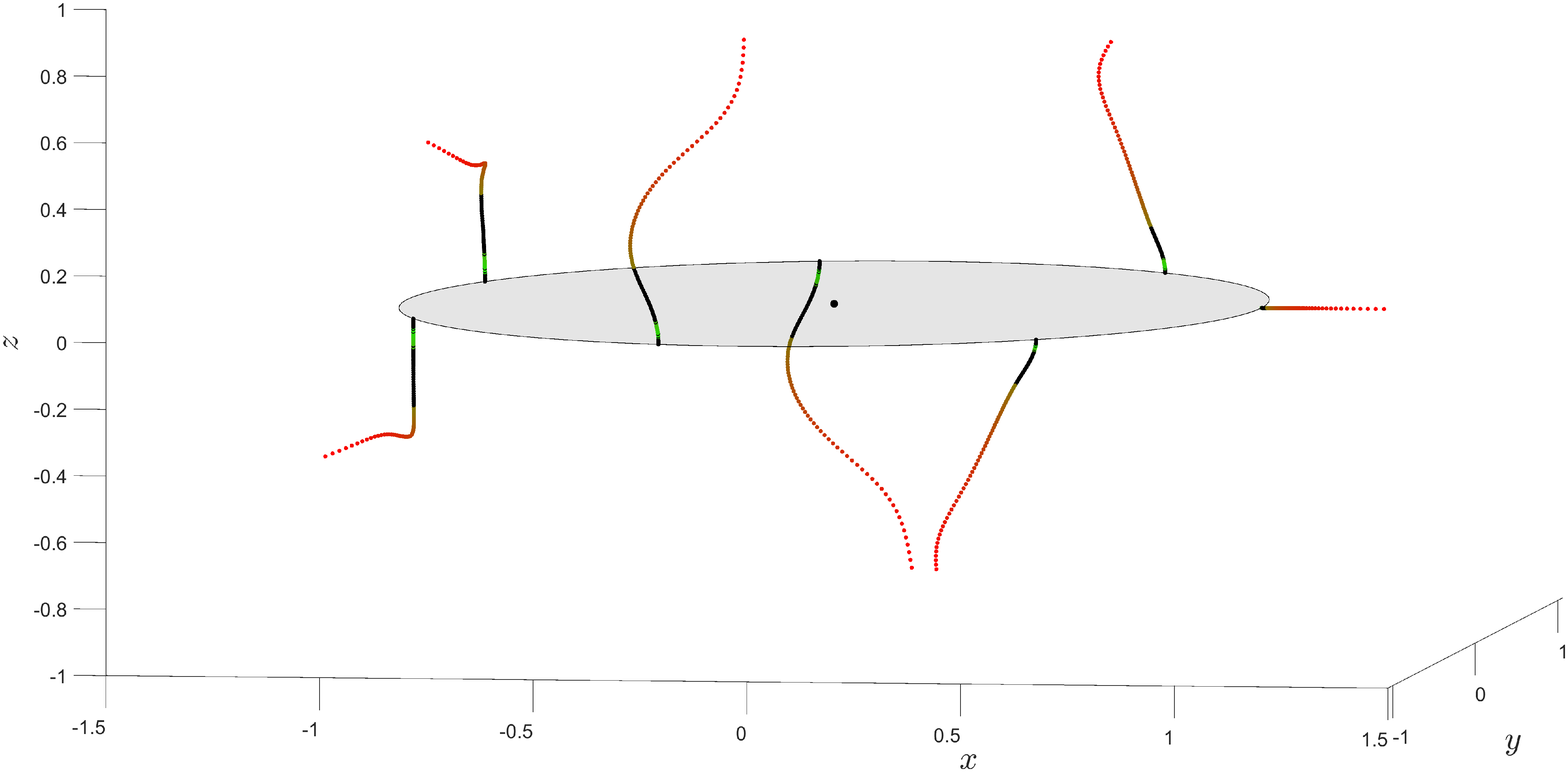}
         \caption{$n=7$, $k=3$}
         \label{fig:3b}
     \end{subfigure}
     \hfill
     \begin{subfigure}[b]{0.49\textwidth}
         \centering
         \includegraphics[width=\textwidth]{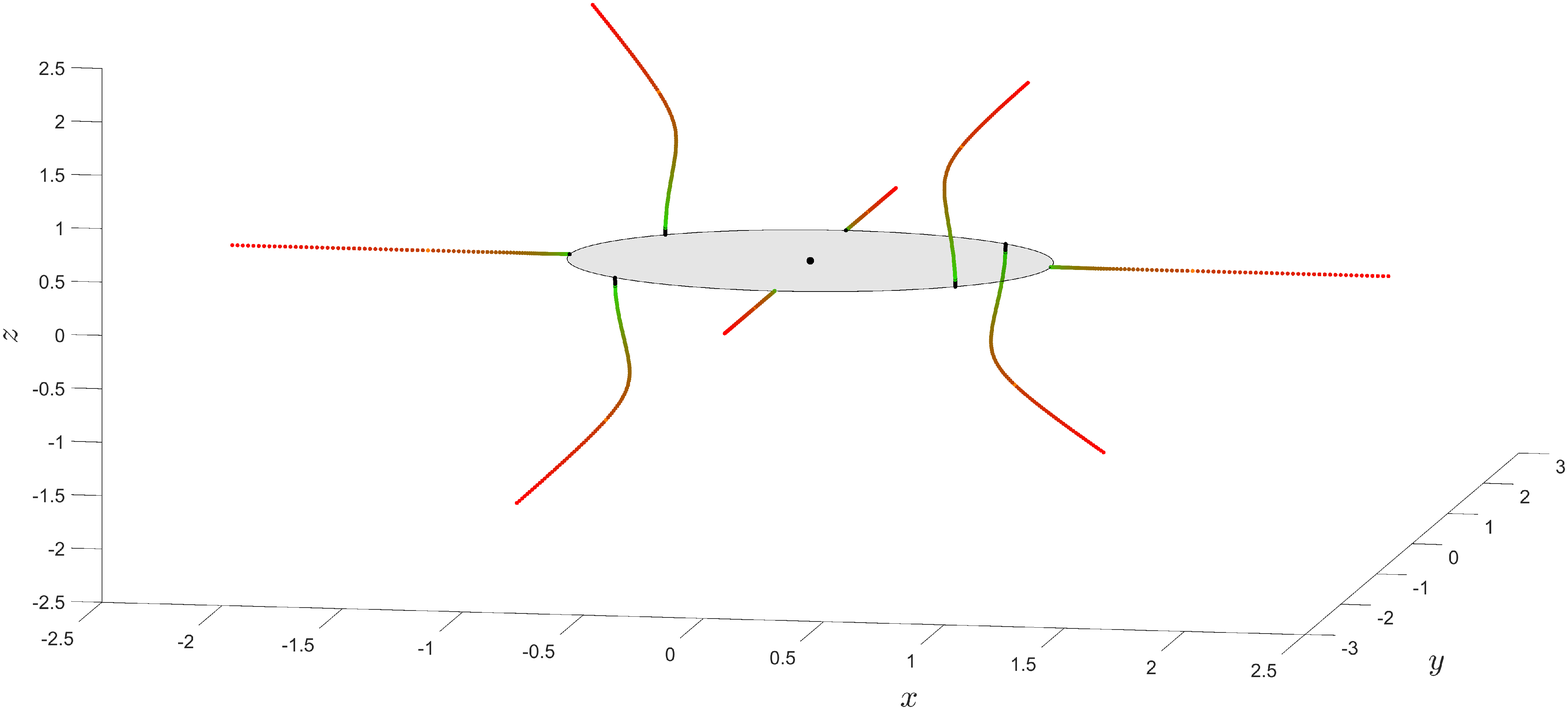}
         \caption{$n=8$, $k=2$}
         \label{fig:4b}
     \end{subfigure}
          \hfill
     \begin{subfigure}[b]{0.49\textwidth}
         \centering
         \includegraphics[width=\textwidth]{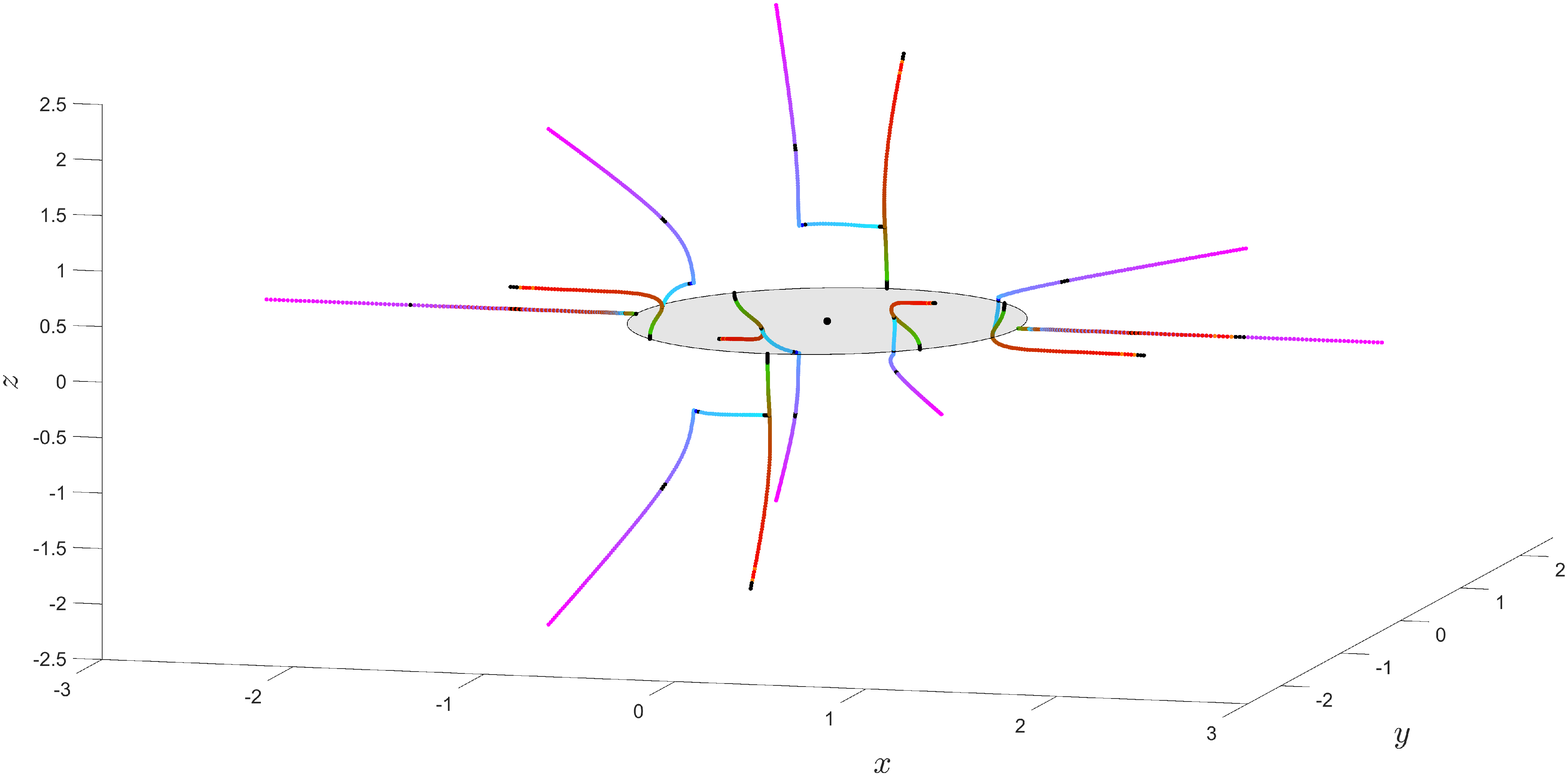}
         \caption{$n=8$, $k=3$}
         \label{fig:5b}
     \end{subfigure}
     \hfill
     \begin{subfigure}[b]{0.49\textwidth}
         \centering
         \includegraphics[width=\textwidth]{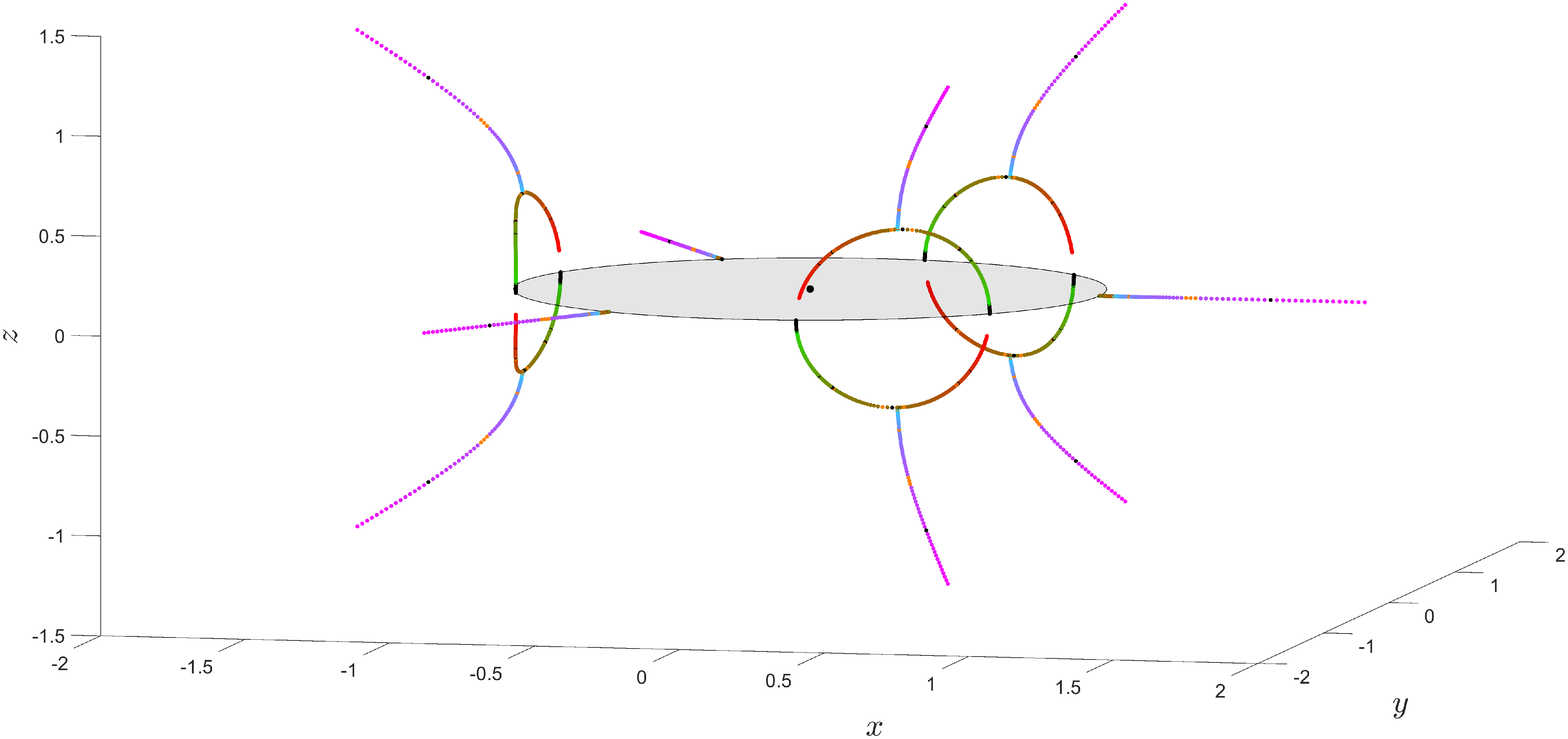}
         \caption{$n=9$, $k=3$}
         \label{fig:6b}
     \end{subfigure}
          \hfill
     \begin{subfigure}[b]{0.49\textwidth}
         \centering
         \includegraphics[width=\textwidth]{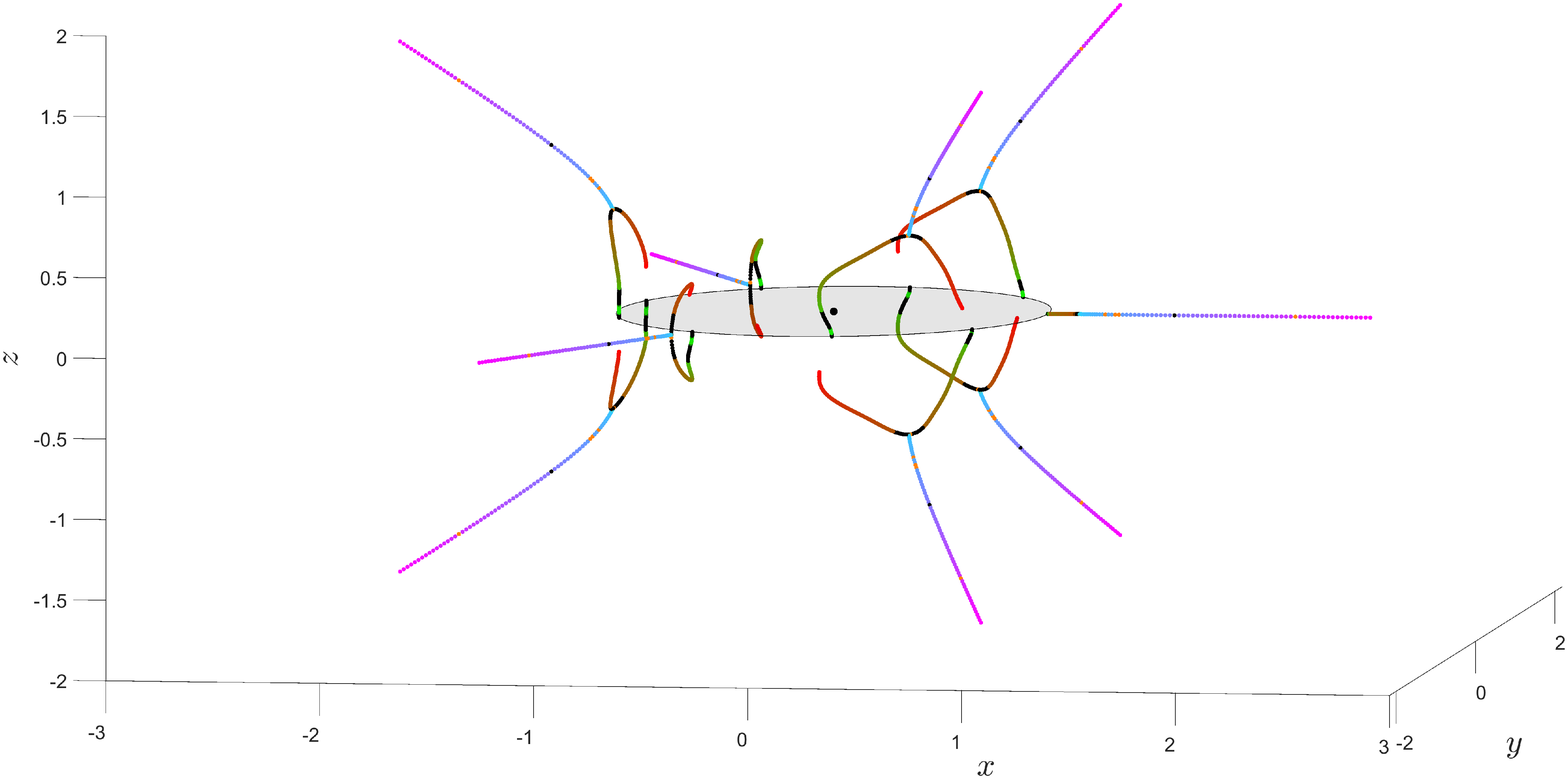}
         \caption{$n=9$, $k=4$}
         \label{fig:7b}
     \end{subfigure}
          \hfill
     \begin{subfigure}[b]{0.49\textwidth}
         \centering
         \includegraphics[width=\textwidth]{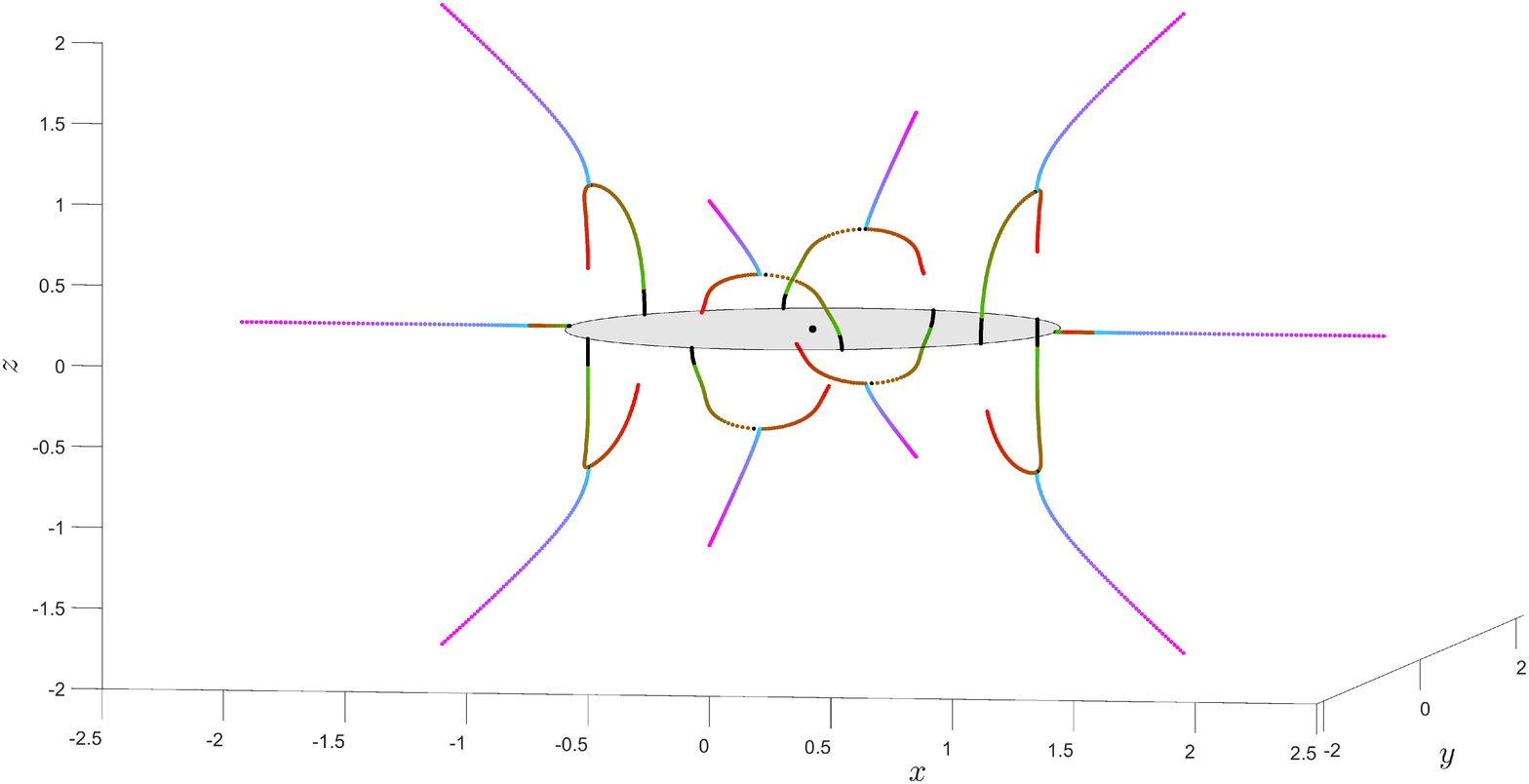}
         \caption{$n=10$, $k=2$}
         \label{fig:8b}
     \end{subfigure}
        \caption{Continuation of equilibria in the family $\mathfrak{F}_{2}^{k}$ for different values of $n$ and $k$. The colour coding in the figure is presented in Remark \ref{coding}.}
        \label{fig:all_plots2}
\end{figure}

\subsection{A computer-assisted validation of the spectra} \label{sec:CAP_frequencies}

%Using the potential $V$ without symmetric constrain.
%
%TO DO: add details of what is the approach, what did we do in practice. 
%
%%non SO(2)-resonant eigenvalue
%
%With the proof of (4) I can give a proof of the existence of periodic
%solutions (Lyapunov families) arising from a point $(u_{0};\mu _{0})\in 
%\mathfrak{F}_{1}^{k}$. Actually, I can give a theoretically proof for the family 
%$\mathfrak{F}_{1}^{k}$ near the $(u;\mu )=(a;s_{k})$. But for reason explained
%before, it is important to have global result.

%\corrc Carlos, do you have a reference for why the eigenvalues must be symmetric and therefore given the radius of uniqueness, the eigenvalue must be purely imaginary? It would be useful to refer to a standard reference here. JP I proved it and I mention that Hamiltonian matrices have the property, it is well known result, it does not need a reference. But for instance, the book of Meyer[9]<<>>

The existence of non-trivial $2\pi /\nu $%
-periodic solutions of (\ref{Equ}) arising from a spatial relative
equilibrium $(u_{0};\mu _{0})$ relies on the validation of the hypotheses of Theorem~\ref{Thm2}; namely, to verify the existence of a $SO(2)$-nonresonant eigenvalue of $L(u_0;\mu_0)$ (see Definition~\ref{Def1}). Now we turn our attention to prove this hypothesis by means of Theorem~\ref{thm:newton-kantorovich}. 
Recall from Remark~\ref{rem:eigs_relations} that $\bar{\lambda},-\lambda ,-\bar{\lambda}$ are
eigenvalues of $L(u_{0};\mu _{0})$ if $\lambda \in \mathbb{C}$ is an
eigenvalue of $L(u_{0};\mu _{0})$. Thus, if we prove the existence of a
unique eigenvalue $\lambda _{0}$ in a neighbourhood $B_{\varepsilon }(i\bar{%
\nu})\subset \mathbb{C}$, then $\lambda _{0}=i\nu_0$ for some $\nu_0 \in \R$ (i.e. $\lambda _{0}$ must be purely imaginary).

Recall that the existence of the relative equilibria $u_0 \in B_{r_0}(\bar u_0)$ is known via a successful application of Theorem~\ref{thm:newton-kantorovich}, where $\bar u_0$ is a numerical solution and $r_0>0$ is the rigorous error bound. The validation of the eigenvalues of $L(u_0;\mu_0)$, which follows the approach \cite{MR3204427}, begins by finding numerically the eigenvalues of $L(\bar u_0;\mu_0)$. Denote by $\bar \lambda_1,\dots,\bar \lambda_{6n}$ the numerical eigenvalues of $L(\bar u_0;\mu_0)$ (computed using the function {\tt eig.m} in MATLAB, which returned the eigenvalues and their corresponding eigenvectors $\bar v_1,\dots,\bar v_{6n}$). Fix $j \in \{1,\dots,6n\}$. Then, in order to obtain local isolation of the eigenpairs $(\lambda_j,v_j) \in \C \times \C^{6n}$, we rescale the eigenvector $v_j$  as follows. Denote by $k=k(j)$ the component of $\bar v_j$ with the largest magnitude, that is 
\[
|(\bar v_j)_k| = \max_{\ell=1,\dots,6n} \left\{ |(\bar v_j)_\ell| \right\}.
\]
Note that $k$ may not be unique. Then, the {\em phase condition} imposed to isolate the eigenpair $(\lambda_j,v_j)$ is $(v_j)_k = (\bar v_j)_k$, where recall that $\bar v_j \in \C^{6n}$ is the numerical approximation for $v_j$. The corresponding zero finding problem is setup in the following way
\begin{equation} \label{eq:eig_proof_map}
F_{\rm eig}(v,\lambda) \bydef 
\begin{pmatrix}
L(u_0;\mu_0)v-\lambda v \\
v \cdot e_k - (\bar v_j)_k
\end{pmatrix} = 0,
\end{equation}
where $\lambda$ and $v$ are the eigenvalue and eigenvector, respectively, and $e_k$ is the $k^{th}$ vector of the canonical basis of $ \R^{6n} $. Without loss of generality, denote by $\lambda_1=\lambda_2=0$ the two zero eigenvalues of $L(u_0;\mu_0)$ due to the action of the group $SO(2)$. 
For each $j \in \{3,\dots,6n\}$, the rigorous enclosure of the eigenpair $(\lambda_j,v_j)$ is obtained by validated the existence of a solution of $F_{\rm eig} = 0$ (where the map is defined in \eqref{eq:eig_proof_map}) using Theorem~\ref{thm:newton-kantorovich}. Denote by $r_j>0$ the radius of the ball $B_{r_j}(\bar \lambda_j,\bar v_j) \subset \C \times \C^{6n}$ which contains the unique eigenpair with $v \cdot e_k = (\bar v_j)_k$, which we denote simply by $(\lambda_j,v_j)$. 

For $j=3,\dots,6n$, denote by 
\[
D_j \bydef \left\{ z \in \C : |z_j-\bar \lambda_j| \le r_j \right\} \subset \C
\]
the disk which contains the true eigenvalue $\lambda_j$. Assume that numerically, two eigenvalues are given by $\pm i \bar \nu_0$, for some $\bar \nu_0>0$. Without loss of generality, denote by $\lambda_{3}$ and $\lambda_{4}$ the true eigenvalues such that 
\[
|\lambda_{3}-i \bar \nu_0| \le r_3 
\quad \text{and} \quad
|\lambda_{4}+i \bar \nu_0| \le r_4.
\]
By unicity, the true eigenvalues satisfy  $\lambda_3= i \nu_0$ and $\lambda_4= -i \nu_0$, for some $\nu_0>0$ since otherwise the disk $D_3$ and $D_4$ would contain more eigenpairs by the comment above (see also Remark~\ref{rem:eigs_relations}). Hence $\lambda_{3}= i \nu_0 \in D_3$ and $\lambda_{4}= i \nu_0 \in D_4$. Denote
\[
\mathcal D \bydef \bigcup_{j=5}^{6n}  D_j
\]
which contains rigorously $\lambda_5,\dots,\lambda_{6n}$. The other four eigenvalues are given by $0,0,\pm i \nu_0$. 
For each spatial relative equilibria rigorously proven in Section~\ref{sec:CAP_rel_eq}, we verified rigorously that $\mathcal D \cap i \nu_0 \tilde \N = \emptyset$ (where $\tilde \N \bydef \{\ell \in \mathbb Z : \ell \ge 2\}$), hence showing rigorously that the eigenvalue $i\nu _{0}$ is a $SO(2)$-nonresonant eigenvalue. All of the computations were carried out in MATLAB using the library INTLAB \cite{Ru99a}.

\vskip0.25cm \textbf{Acknowledgements.} JP.L. was partially supported by NSERC Discovery Grant. K.C. was partially supported by an ISM-CRM Undergraduate Summer Scholarship.  C.G.A was partially supported by UNAM-PAPIIT project IA100121.


\begin{thebibliography}{10}

\bibitem{AlPe02}
Felipe Alfaro~Aguilar and Ernesto P{\'e}rez-Chavela.
\newblock Relative equilibria in the charged n-body problem.
\newblock {\em The Canadian Applied Mathematics Quarterly}, 10, 01 2002.

\bibitem{MR3917433}
Istv\'{a}n Bal\'{a}zs, Jan~Bouwe van~den Berg, Julien Courtois, J\'{a}nos
  Dud\'{a}s, Jean-Philippe Lessard, Anett V\"{o}r\"{o}s-Kiss, J.~F. Williams,
  and Xi~Yuan Yin.
\newblock Computer-assisted proofs for radially symmetric solutions of {PDE}s.
\newblock {\em J. Comput. Dyn.}, 5(1-2):61--80, 2018.

\bibitem{MR3204427}
Roberto Castelli and Jean-Philippe Lessard.
\newblock A method to rigorously enclose eigenpairs of complex interval
  matrices.
\newblock In {\em Applications of mathematics 2013}, pages 21--31. Acad. Sci.
  Czech Repub. Inst. Math., Prague, 2013.

\bibitem{MR726510}
Ian Davies, Aubrey Truman, and David Williams.
\newblock Classical periodic solution of the equal-mass {$2n$}-body problem,
  {$2n$}-ion problem and the {$n$}-electron atom problem.
\newblock {\em Phys. Lett. A}, 99(1):15--18, 1983.

\bibitem{MR2338393}
Sarah Day, Jean-Philippe Lessard, and Konstantin Mischaikow.
\newblock Validated continuation for equilibria of {PDE}s.
\newblock {\em SIAM J. Numer. Anal.}, 45(4):1398--1424 (electronic), 2007.

\bibitem{MR4032352}
Marco Fenucci and \`Angel Jorba.
\newblock Braids with the symmetries of {P}latonic polyhedra in the {C}oulomb
  {$(N+1)$}-body problem.
\newblock {\em Commun. Nonlinear Sci. Numer. Simul.}, 83:105105, 12, 2020.

\bibitem{MR2832692}
C.~Garc\'{\i}a-Azpeitia and J.~Ize.
\newblock Global bifurcation of polygonal relative equilibria for masses,
  vortices and d{NLS} oscillators.
\newblock {\em J. Differential Equations}, 251(11):3202--3227, 2011.

\bibitem{MR3007103}
C.~Garc\'{\i}a-Azpeitia and J.~Ize.
\newblock Global bifurcation of planar and spatial periodic solutions from the
  polygonal relative equilibria for the {$n$}-body problem.
\newblock {\em J. Differential Equations}, 254(5):2033--2075, 2013.

\bibitem{HLM}
Allan Hungria, Jean-Philippe Lessard, and J.~D. Mireles~James.
\newblock Rigorous numerics for analytic solutions of differential equations:
  the radii polynomial approach.
\newblock {\em Math. Comp.}, 85(299):1427--1459, 2016.

\bibitem{MR1984999}
Jorge Ize and Alfonso Vignoli.
\newblock {\em Equivariant degree theory}, volume~8 of {\em De Gruyter Series
  in Nonlinear Analysis and Applications}.
\newblock Walter de Gruyter \& Co., Berlin, 2003.

\bibitem{MR910499}
H.~B. Keller.
\newblock {\em Lectures on numerical methods in bifurcation problems},
  volume~79 of {\em Tata Institute of Fundamental Research Lectures on
  Mathematics and Physics}.
\newblock Published for the Tata Institute of Fundamental Research, Bombay,
  1987.
\newblock With notes by A. K. Nandakumaran and Mythily Ramaswamy.

\bibitem{MR255046}
R.~Krawczyk.
\newblock Newton-{A}lgorithmen zur {B}estimmung von {N}ullstellen mit
  {F}ehlerschranken.
\newblock {\em Computing (Arch. Elektron. Rechnen)}, 4:187--201, 1969.

\bibitem{LAFAVE20131029}
Tim LaFave.
\newblock Correspondences between the classical electrostatic thomson problem
  and atomic electronic structure.
\newblock {\em Journal of Electrostatics}, 71(6):1029 -- 1035, 2013.

\bibitem{MR1140006}
Kenneth~R. Meyer and Glen~R. Hall.
\newblock {\em Introduction to {H}amiltonian dynamical systems and the
  {$N$}-body problem}, volume~90 of {\em Applied Mathematical Sciences}.
\newblock Springer-Verlag, New York, 1992.

\bibitem{Moeckel_1990}
Richard Moeckel.
\newblock On central configurations.
\newblock {\em Mathematische Zeitschrift}, 205(1):499--517, sep 1990.

\bibitem{Moeckel_1995}
Richard Moeckel and Carles Sim{\'{o}}.
\newblock Bifurcation of spatial central configurations from planar ones.
\newblock {\em {SIAM} Journal on Mathematical Analysis}, 26(4):978--998, jul
  1995.

\bibitem{MR657002}
R.~E. Moore.
\newblock A test for existence of solutions to nonlinear systems.
\newblock {\em SIAM J. Numer. Anal.}, 14(4):611--615, 1977.

\bibitem{MR0231516}
Ramon~E. Moore.
\newblock {\em Interval analysis}.
\newblock Prentice-Hall Inc., Englewood Cliffs, N.J., 1966.

\bibitem{MR2003792}
F.~J. Mu\~{n}oz Almaraz, E.~Freire, J.~Gal\'{a}n, E.~Doedel, and
  A.~Vanderbauwhede.
\newblock Continuation of periodic orbits in conservative and {H}amiltonian
  systems.
\newblock {\em Phys. D}, 181(1-2):1--38, 2003.

\bibitem{Ru99a}
{S.M.} Rump.
\newblock {INTLAB - INTerval LABoratory}.
\newblock In Tibor Csendes, editor, {\em {Developments~in~Reliable Computing}},
  pages 77--104. Kluwer Academic Publishers, Dordrecht, 1999.
\newblock http://www.ti3.tu-harburg.de/rump/.

\end{thebibliography}
\end{document}